\documentclass{article}
\usepackage{amsmath, amssymb, amsthm, graphicx, hyperref}
\usepackage{enumitem}
\usepackage{listings}
\usepackage{xcolor}
\usepackage{url}
\newtheorem{theorem}{Theorem}[section]
\newtheorem{proposition}[theorem]{Proposition}
\newtheorem{corollary}[theorem]{Corollary}
\newtheorem{lemma}[theorem]{Lemma}
\newtheorem{example}[theorem]{Example}
\theoremstyle{remark}
\newtheorem{remark}[theorem]{Remark}
\newtheorem{definition}[theorem]{Definition}

\title{Operator Michelson Contrast: Logarithmic Parametrisation, Subadditivity, and a Sharp Noncommutativity Threshold}
\author{Irina Nikolaeva\thanks{Lean Technologies Laboratory, Kazan, Tatarstan, 420000, Krasnokokshayskaia 189a-202} \and Andrei Novikov\thanks{Novikov Laboratories, Kazan, Tatarstan, 420124, Absaliamova 14-124}}
\date{}

\providecommand{\QLC}{\operatorname{OLC}}
\providecommand{\QMC}{\operatorname{OMC}}
\providecommand{\MC}{\operatorname{MC}}
\providecommand{\Emax}{E_{\max}}
\providecommand{\Emin}{E_{\min}}
\providecommand{\OLC}{\QLC}

\providecommand{\norm}[1]{\lVert#1\rVert}

\DeclareMathOperator{\Tr}{Tr}
\DeclareMathOperator{\diag}{diag}

\begin{document}
\maketitle

\begin{abstract}
For a bounded invertible operator $A$ with condition number $\kappa(A)=\|A\|\|A^{-1}\|$, we study the logarithmic form $\QLC(A)=\tfrac12\ln\kappa(A)$ of the operator Michelson contrast. For positive definite matrices, the product bound $\kappa(XY)\le\kappa(X)\kappa(Y)$ is attained precisely when the factors have a common maximal eigenvector and a common minimal eigenvector. This forces commutativity in dimensions at most three, and dimension four is the first admitting noncommuting equality cases. For bounded positive invertible operators on arbitrary Hilbert spaces, a corresponding equality criterion uses common approximate extremal eigenvectors. Equality passes to each pair of factors in a fixed tensor decomposition. We identify logarithmic contrast with projective Thompson distance from the identity ray and treat its continuity, contraction, and behaviour under sums, products, and powers. An application to sequential positive filters identifies saturation of the ratio of their extremal success probabilities. For density matrices, we give the exact qubit entropy relation, sharp qutrit entropy and purity ranges at fixed condition number, dimension-dependent entropy bounds, and a log-spectral Fourier description. We also derive the qubit quantum Jensen--Shannon formula and give explicit higher-dimensional counterexamples to determination by the input contrasts.
\end{abstract}



\section{Introduction}

The classical Michelson contrast for two positive levels $0<m\le M$ is defined as
\[
\MC(M,m) = \frac{M-m}{M+m}.
\]
It satisfies the elementary identity
\[
\operatorname{arctanh}(\MC(M,m)) = \tfrac12\ln\frac{M}{m}.
\]
Thus the Michelson contrast may be viewed as a bounded parametrisation of the logarithmic ratio $M/m$. In logarithmic image processing, logarithmic representations of contrast have been developed; in particular, the logarithmic additive contrast (LAC) is closely related to the logarithmic ratio and the multiplicative structure of intensity levels. The works of Jourlin, Pinoli, and collaborators \cite{jourlin1988,jourlin1989} provide a classical foundation for this approach.

In a recent paper \cite{abed2024}, we introduced an operator analogue of the Michelson contrast for positive invertible operators. Here the logarithmic coordinate is also defined on arbitrary invertible bounded operators through the condition number. For any invertible positive operator $A\in B(H)_{++}$ (where $B(H)$ denotes bounded operators on a Hilbert space $H$ of any dimension), define
\[
\Delta(A) = \frac{\|A\|-\|A^{-1}\|^{-1}}{\|A\|+\|A^{-1}\|^{-1}} = \frac{\kappa(A)-1}{\kappa(A)+1}, \qquad \kappa(A):=\|A\|\|A^{-1}\|.
\]
This operator Michelson contrast extends the classical formula by replacing the two scalar levels with the spectral supremum and infimum of the operator. It depends only on the condition number $\kappa(A)=\sup\sigma(A)/\inf\sigma(A)$, which is intrinsically dimension-independent.

The present paper investigates the logarithmic coordinate naturally associated with this operator Michelson contrast, in both finite and infinite dimensions:
\[
\QLC(A) := \operatorname{arctanh}(\Delta(A)) = \tfrac12\ln\kappa(A).
\]
We call this quantity the \emph{Operator Logarithmic Contrast} (OLC); the name is proposed in this work and is not an established term in the literature. It is the natural logarithmic coordinate for the operator Michelson contrast of \cite{abed2024}, and requires no physical interpretation to be well defined: it is a functional on invertible operators, studied here purely as such. Where the paper later specializes to density operators (Section \ref{sec:applications}), that specialization is presented as a worked example class -- a convenient, well-studied family of positive operators on which the general theory can be illustrated concretely -- rather than as the paper's central claim.

This logarithmic parametrisation reveals several new properties, which we develop first in complete generality for any Hilbert space, then specialize to finite dimensions where finer structure becomes visible. First, unlike the classical logarithmic ratio which is additive under multiplication of scalar ratios, the operator version satisfies only a subadditivity inequality, valid universally:
\[
\QLC(AB) \le \QLC(A) + \QLC(B),
\]
which follows from the submultiplicativity of the condition number and holds for all invertible operators in any Hilbert space. This inequality extends to limit-theoretic contexts: for convergent infinite products and Ces\`aro means of operator sequences, analogous inequalities hold under operator-norm convergence to an invertible limit (Section \ref{sec:products}). Second, we prove a geometric identification: for positive invertible operators, $\QLC(A) = \delta_T([I],[A])$, where $\delta_T$ is the projective Thompson metric on the positive cone; this immediately yields a Lipschitz-continuity property for OLC that we make explicit below. Third, the equality case in the subadditivity inequality leads to an alignment condition for the extremal eigenspaces: $n=4$ is the smallest dimension in which noncommuting equality cases can occur, and we show the corresponding rigidity survives among pure tensor products with a common decomposition into factors of dimension at most three. Applied to quantum information, these results give an exact closed-form relation between OMC and von Neumann entropy for qubits, and different entropies at a fixed condition number in every dimension at least three, including sharp bounds in dimension three.

For example, $X=\diag(2,1)$ and $Y=\diag(1,2)$ commute, but $XY=2I$ and their product inequality is strict. Thus the issue is whether the same directions realize both factors' largest and smallest spectral scales. Our equality criterion describes these directions and explains why two remaining dimensions are needed for noncommutativity.

The paper is organised as follows. Section 2 fixes notation and definitions, extending OLC to all invertible operators via the polar decomposition. Section 3 recalls the Thompson metric, proves the geometric identification theorem, and derives the continuity properties of OLC. Section 4 contains the linear-algebraic core: subadditivity, the equality case and dimensional threshold, tensor-product behaviour, power laws, contraction, sum inequalities, trace estimates, and double-sided products. Section 5 works out density operators as an example class. Section 6 discusses structural interpretations, open questions, and concludes.

\section{Definitions and Notation}

Let $H$ be a complex Hilbert space and $B(H)$ the algebra of all bounded operators. Let $B(H)_{++}$ and $B(H)_+$ denote the positive invertible and positive (including singular) operators, respectively.

For invertible $A\in B(H)$, the condition number is $\kappa(A):=\|A\|\|A^{-1}\|$. We recall the standard facts (see \cite{dixmier1969}):
\[
\kappa(A)\ge1,\qquad \kappa(AB)\le\kappa(A)\kappa(B),
\]
with equality in the first iff $A$ is a scalar multiple of a unitary operator. For $A\in B(H)_{++}$, $\kappa(A) = \sup\sigma(A)/\inf\sigma(A)$.

\subsection{The Operator Michelson Contrast}
In \cite{abed2024}, the operator Michelson contrast for positive operators was defined as
\[
\Delta(A) = \begin{cases} \dfrac{\|A\|-\|A^{-1}\|^{-1}}{\|A\|+\|A^{-1}\|^{-1}}, & A\in B(H)_{++},\\[2mm] 1, & A \text{ singular}. \end{cases}
\]
For invertible positive $A$, this equals $(\kappa(A)-1)/(\kappa(A)+1)$.

\subsection{Logarithmic Parametrisation}

\begin{definition}\label{def:qmc}
For any invertible $A\in B(H)$, define the \emph{Operator Michelson Contrast} (OMC) as the functional extension of $\Delta$:
\[
\QMC(A) := \frac{\kappa(A)-1}{\kappa(A)+1},
\]
and its logarithmic parametrisation, the \emph{Operator Logarithmic Contrast} (OLC), by
\[
\QLC(A) := \operatorname{arctanh}(\QMC(A)) = \tfrac12\ln\kappa(A).
\]
For singular nonzero operators, we may extend these quantities by the conventions $\QMC(A)=1$ and $\QLC(A)=+\infty$. The zero operator is excluded from this extended convention.
\end{definition}

For positive invertible $A$, this coincides with $\Delta(A)$ from \cite{abed2024}. By the geometric identification of Theorem \ref{thm:thompson} below, for positive invertible $A$, $\QLC(A)=\delta_T([I],[A])$.

\subsection{Extension to All Invertible Operators via Polar Decomposition}
Since $\kappa(A)=\kappa(|A|)$ for any invertible $A$ (polar decomposition $A=U|A|$), we have
\[
\QLC(A) = \tfrac12\ln\kappa(A) = \tfrac12\ln\kappa(|A|) = \QLC(|A|) = \delta_T([I],[|A|]).
\]
Thus the logarithmic contrast functional OLC on arbitrary invertible operators is obtained by pulling back the projective Thompson distance-to-identity functional through the absolute value map $A\mapsto|A|$. This is a natural extension of the functional, not of the metric itself, since the projective Thompson metric is originally defined only on the positive cone.

\begin{remark}\label{rem:polar}
This reduction is valid only for statements about $\kappa$, OMC, or OLC themselves; it does not transfer commutativity statements. If $X=\sigma_x$, $Y=\sigma_y$ are Pauli matrices, then $|X|=|Y|=I$ commute trivially, while $XY=i\sigma_z=-YX$ are maximally noncommuting. The dimensional-threshold result of Theorem \ref{thm:threshold} is intrinsically a statement about positive operators $X,Y$, using their own extremal eigenspaces $\Emax,\Emin$; it does not extend to general invertible operators via $|X|,|Y|$. See also Remark \ref{rem:entanglement}.
\end{remark}

Both quantities are invariant under positive scaling: for $\lambda>0$, $\QMC(\lambda A)=\QMC(A)$, $\QLC(\lambda A)=\QLC(A)$.

\section{Preliminaries: The Thompson Metric, Geometric Interpretation, and Continuity}

The Thompson metric on $B(H)_{++}$ (see \cite{thompson1963,corach1994}) is
\[
d_T(A,B) := \|\log(A^{-1/2}BA^{-1/2})\|.
\]
The projective version on rays $[A]=\{\lambda A:\lambda>0\}$ is
\[
\delta_T([A],[B]) := \inf_{\lambda>0} d_T(A,\lambda B).
\]
Further background on the Thompson metric and its relation to projective metrics can be found in \cite{lawson2001}.

\begin{theorem}\label{thm:thompson}
For any $A\in B(H)_{++}$,
\[
\delta_T([I],[A]) = \tfrac12\ln\kappa(A). \tag{1}
\]
\end{theorem}
\begin{proof}
Let $m=\inf\sigma(A)$, $M=\sup\sigma(A)$. Then $\kappa(A)=M/m$. For any $\lambda>0$,
\[
d_T(I,\lambda A) = \|\log(\lambda A)\| = \max\{|\log(\lambda m)|,|\log(\lambda M)|\}.
\]
Minimising over $\lambda>0$ gives $\lambda_*=(mM)^{-1/2}$. The minimum equals $\tfrac12\ln(M/m)=\tfrac12\ln\kappa(A)$.
\end{proof}

Thus, for positive invertible operators, $\QLC(A)=\tfrac12\ln\kappa(A)$ is exactly the distance from the identity ray to the ray generated by $A$ in the projective Thompson metric.

\subsection{Continuity of OLC}

Theorem \ref{thm:thompson} has an immediate and useful consequence: since OLC \emph{is} a distance-to-a-fixed-point functional on a metric space, it inherits a Lipschitz bound for free from the triangle inequality. We record this, together with a companion statement showing that the singular convention $\QLC=+\infty$ of Definition \ref{def:qmc} is not an arbitrary choice but the actual limiting value of OLC as an invertible operator degenerates.

\begin{corollary}[1-Lipschitz continuity w.r.t.\ the Thompson metric]\label{cor:lipschitz}
For all $A,B\in B(H)_{++}$,
\[
|\QLC(A)-\QLC(B)| \le \delta_T([A],[B]).
\]
\end{corollary}
\begin{proof}
By Theorem \ref{thm:thompson}, $\QLC(A)=\delta_T([I],[A])$ and $\QLC(B)=\delta_T([I],[B])$. The reverse triangle inequality for the metric on rays gives $|\delta_T([I],[A])-\delta_T([I],[B])|\le\delta_T([A],[B])$.
\end{proof}

On the set $A,B\ge aI$ with $a>0$, both spectral endpoints vary by at most $\|A-B\|$. The scalar logarithm is $1/a$-Lipschitz on $[a,\infty)$, so
\[
 |\QLC(A)-\QLC(B)|\le a^{-1}\|A-B\|.
\]
A bound on condition numbers alone does not give this conclusion: $tI_2$ and $t\diag(2,1)$ approach each other as $t\downarrow0$, but their OLC difference is always $\tfrac12\ln2$.

\begin{proposition}[Continuity at singular limits]\label{prop:singular-cont}
Let $A_n\in B(H)$ be invertible with $A_n\to A$ in operator norm, where $A\neq0$ is singular (not necessarily positive). Then $\kappa(A_n)\to+\infty$, hence $\QLC(A_n)\to+\infty=\QLC(A)$ under the extended convention of Definition \ref{def:qmc}.
\end{proposition}
\begin{proof}
Norm convergence gives $\|A_n\|\to\|A\|$, and $\|A\|>0$ since $A\neq0$; in particular $\|A_n\|$ is bounded away from $0$ for large $n$. It remains to show $\|A_n^{-1}\|\to\infty$. Suppose not: then along a subsequence $\|A_n^{-1}\|\le M$ for some $M<\infty$. Since $A_n\to A$, for $n$ large enough $\|A_n-A\|<1/M$, and the standard Neumann-series perturbation bound for invertible operators shows that $A = A_n - (A_n-A) = A_n(I - A_n^{-1}(A_n-A))$ is invertible whenever $\|A_n^{-1}(A_n-A)\|\le\|A_n^{-1}\|\|A_n-A\|<1$, contradicting the singularity of $A$. Hence $\|A_n^{-1}\|\to\infty$ along every subsequence, so $\kappa(A_n)=\|A_n\|\|A_n^{-1}\|\to\infty$, and $\QLC(A_n)=\tfrac12\ln\kappa(A_n)\to+\infty$.
\end{proof}

Thus the extended convention $\QLC(A)=+\infty$ for nonzero singular $A$ is exactly the value forced by continuity from the invertible operators, not an independent stipulation.

\begin{remark}[Spectral endpoints in infinite dimensions]\label{rem:infinite-dim}
A positive invertible operator need not have eigenvectors at its spectral endpoints. Lemma \ref{lem:align} below formulates equality by approximate eigenvectors and gives the finite-dimensional result as a special case. If an endpoint of one factor is isolated and its eigenspace is finite-dimensional, the spectral gap forces the corresponding approximate extremal vectors to approach that eigenspace. A convergent subsequence then yields a common eigenvector for both factors at the endpoints in question.
\end{remark}

\section{Core Linear-Algebraic Statements}\label{sec:products}

\subsection{Subadditivity and Composition Laws}

\begin{theorem}[Subadditivity in infinite and finite dimensions]\label{thm:subadd}
For any invertible $A,B\in B(H)$ (where $H$ is any complex Hilbert space, finite or infinite-dimensional),
\[
\QLC(AB) \le \QLC(A)+\QLC(B). \tag{2}
\]
Equivalently,
\[
\QMC(AB) \le \frac{\QMC(A)+\QMC(B)}{1+\QMC(A)\QMC(B)}. \tag{3}
\]
\end{theorem}
\begin{proof}
The proof is identical in finite and infinite dimensions: use the submultiplicativity $\kappa(AB)\le\kappa(A)\kappa(B)$, which holds for any bounded invertible operators. Taking logarithms gives (2); applying $\tanh$ yields (3).
\end{proof}

Thus OLC is subadditive rather than additive. Strict inequality can occur even for commuting operators if their extremal spectral directions are not aligned; see below.

\begin{remark}[Convergence and infinite products in limit-theoretic setting]
Theorem \ref{thm:subadd} immediately implies: if $A_1, A_2, A_3, \ldots\in B(H)_{++}$ is a sequence of invertible positive operators and the partial products $P_N:=A_1A_2\cdots A_N$ converge in operator norm to an invertible limit $P$, then
\[
\QLC(P) \le \liminf_{N\to\infty}\sum_{k=1}^N\QLC(A_k).
\]
If the series $\sum_{k=1}^\infty\QLC(A_k)$ converges, then $\QLC(P)\le\sum_{k=1}^\infty\QLC(A_k)$. This applies equally in infinite-dimensional Hilbert spaces: the operator-norm convergence and the continuity of OLC (Corollary \ref{cor:lipschitz}, Proposition \ref{prop:singular-cont}) require no assumption on the dimension of $H$. 

Similarly, for Ces\`aro means $S_N=\tfrac1N\sum_{k=1}^N A_k$ of a sequence of positive invertible operators, the sum inequality (9) gives $\QLC(S_N)\le\max_k\QLC(A_k)$, and if $S_N$ converges in operator norm to an invertible $A$, then $\QLC(A)\le\sup_k\QLC(A_k)$ by continuity.
\end{remark}

\subsection{Equality Case and Dimensional Threshold}\label{sec:threshold}

For bounded invertible $X,Y$, equality in (2) is equivalent to simultaneous equality in the two norm bounds for $XY$ and $Y^{-1}X^{-1}$:
\[
\|XY\| = \|X\|\|Y\|, \qquad \|(XY)^{-1}\| = \|X^{-1}\|\|Y^{-1}\|. \tag{4}
\]

\begin{lemma}[Extremal subspace alignment]\label{lem:align}
Let \(X,Y\in B(H)_{++}\). Equality holds in (2) if and only if there are unit vectors \(u_j,v_j\) such that
\begin{align*}
 \norm{(X-M_X I)u_j}+\norm{(Y-M_Y I)u_j}&\longrightarrow 0,\\
 \norm{(X-m_X I)v_j}+\norm{(Y-m_Y I)v_j}&\longrightarrow 0.
\end{align*}
In finite dimensions these conditions are equivalent to
\[
 \Emax(X)\cap\Emax(Y)\ne\{0\},\qquad
 \Emin(X)\cap\Emin(Y)\ne\{0\},
\]
where \(\Emax\) and \(\Emin\) are the respective eigenspaces.
\end{lemma}

\begin{proof}
For a positive operator \(B\) with \(b=\norm{B}\) and a unit vector \(w\), functional calculus yields
\[
 \norm{(bI-B)w}^{2}\leq b^2-\norm{Bw}^{2}.
\]
Suppose \(\norm{XY}=M_XM_Y\), and take unit vectors \(u_j\) with \(\norm{XYu_j}\to M_XM_Y\). Since \(\norm{XYu_j}\leq M_X\norm{Yu_j}\leq M_XM_Y\), we have \(\norm{Yu_j}\to M_Y\) and then \((Y-M_YI)u_j\to0\). The identity \(XYu_j-M_YXu_j=X(Y-M_YI)u_j\) gives \(\norm{Xu_j}\to M_X\), so \((X-M_XI)u_j\to0\) as well. Conversely, the two approximate eigenvector conditions give \((XY-M_XM_YI)u_j\to0\), establishing equality for the norm.

Apply the same argument to \(Y^{-1}X^{-1}\), using \(M_{X^{-1}}=1/m_X\) and \(M_{Y^{-1}}=1/m_Y\). For example,
\[
 (X^{-1}-m_X^{-1}I)v_j=-m_X^{-1}X^{-1}(X-m_XI)v_j,
\]
and multiplication by the bounded inverse of \(-m_X^{-1}X^{-1}\) gives the converse implication. The result for the inverse norm follows. In finite dimensions take convergent subsequences of \(u_j\) and \(v_j\).
\end{proof}

The approximate formulation is essential. If \(X=Y\) is multiplication by \(1+t\) on \(L^2(0,1)\), equality holds by the spectral theorem and neither endpoint is an eigenvalue. At an isolated endpoint with finite-dimensional eigenspace for one of the factors, the approximate sequence is asymptotically contained in that eigenspace by the spectral gap. A norm-convergent subsequence then yields a common eigenvector at that endpoint.

\begin{theorem}[Dimensional threshold for equality]\label{thm:threshold}
Let $X,Y\in M_n(\mathbb{C})_{++}$. If $\QLC(XY)=\QLC(X)+\QLC(Y)$, then $XY=YX$ for $n\le3$. Moreover, $n=4$ is the smallest dimension in which noncommuting equality cases can occur.
\end{theorem}
\begin{proof}
For $n=1$ trivial.

\emph{Case $n=2$.} Let $X=\operatorname{diag}(\lambda_1,\lambda_2)$, $\lambda_1\ge\lambda_2>0$, and $Y=\begin{pmatrix}a&b\\\bar b&c\end{pmatrix}>0$. If $\lambda_1=\lambda_2$, $X$ is scalar and commutes. Assume $\lambda_1>\lambda_2$. By Lemma \ref{lem:align}, there exists $u\in\Emax(X)\cap\Emax(Y)$. Since $\Emax(X)=\operatorname{span}\{e_1\}$, $u$ is proportional to $e_1$, so $Ye_1=\mu e_1$, giving $b=0$. Similarly, the minimal alignment gives $\Emin(X)\cap\Emin(Y)\neq\{0\}$, so $e_2$ is an eigenvector of $Y$. Hence $Y$ is diagonal and commutes with $X$.

\emph{Case $n=3$.} Let $X=\operatorname{diag}(\lambda_1,\lambda_2,\lambda_3)$, $\lambda_1\ge\lambda_2\ge\lambda_3>0$, and $Y>0$. If all eigenvalues of $X$ coincide, then $X$ is scalar. The remaining cases follow below.

\emph{Subcase $\lambda_1>\lambda_2>\lambda_3$.} $\Emax(X)=\operatorname{span}\{e_1\}$ is one-dimensional; alignment gives $Ye_1=\mu_1e_1$. Similarly $\Emin(X)=\operatorname{span}\{e_3\}$ gives $Ye_3=\mu_3e_3$. Since $Y$ is self-adjoint, $\operatorname{span}\{e_1\}$ and $\operatorname{span}\{e_3\}$ are $Y$-invariant, hence so is their orthogonal complement $\operatorname{span}\{e_2\}$; thus $Y$ is diagonal in the basis $\{e_1,e_2,e_3\}$.

\emph{Subcase $\lambda_1=\lambda_2>\lambda_3$.} Let $E=\operatorname{span}\{e_1,e_2\}$, $F=\operatorname{span}\{e_3\}$. Minimal alignment gives $e_3\in\Emin(Y)$ (since $\Emin(X)=F$), so $Ye_3=\mu_3e_3$; thus $F$ is $Y$-invariant, hence so is $E=F^\perp$. On $E$, $X$ is scalar ($\lambda_1I_E$), so $Y$ commutes with $X$ on $E$; on $F$, $X$ is scalar ($\lambda_3I_F$), so $Y$ commutes there too. Therefore $XY=YX$.

\emph{Subcase $\lambda_1>\lambda_2=\lambda_3$.} Symmetric argument using maximal alignment.

Thus for $n\le3$, equality forces commutativity.

\emph{Noncommuting example for $n=4$.} Let
\[
X=\begin{pmatrix}4&0&0&0\\0&3&0&0\\0&0&1&0\\0&0&0&1\end{pmatrix},\qquad
Y=\begin{pmatrix}4&0&0&0\\0&2&1&0\\0&1&2&0\\0&0&0&1\end{pmatrix}.
\]
Then $\kappa(X)=4$, and $\sigma(Y)=\{4,3,1,1\}$, so $\kappa(Y)=4$. We have
\[
XY=\begin{pmatrix}16&0&0&0\\0&6&3&0\\0&1&2&0\\0&0&0&1\end{pmatrix}.
\]
Let $B=\begin{pmatrix}6&3\\1&2\end{pmatrix}$. Then $\|B\|\le\|B\|_F=\sqrt{50}<16$, and $B^{-1}=\tfrac19\begin{pmatrix}2&-3\\-1&6\end{pmatrix}$, $\|B^{-1}\|\le\|B^{-1}\|_F=\sqrt{50}/9<1$. Since $XY$ has the scalar blocks $16$ and $1$, we get $\|XY\|=16$ and $\|(XY)^{-1}\|=1$. Hence $\kappa(XY)=16=4\cdot4$, so equality in (2) holds. However, the blocks on $e_2,e_3$ satisfy
\[
\operatorname{diag}(3,1)\begin{pmatrix}2&1\\1&2\end{pmatrix} \neq \begin{pmatrix}2&1\\1&2\end{pmatrix}\operatorname{diag}(3,1),
\]
so $X$ and $Y$ do not commute. This establishes that noncommuting equality cases exist for $n=4$, and hence for all $n\ge4$ by taking direct sums with identity blocks.
\end{proof}

\begin{remark}[Common reducing subspace]
If both factors are nonscalar and equality is realized by actual extremal eigenvectors, the common maximal and minimal eigenvectors are orthogonal. Each spans a common reducing line. The restrictions to their common orthogonal complement determine whether the pair commutes. This proves rigidity whenever the complement has dimension at most one. An infinite-dimensional complement can carry noncommuting operators even when both spectral endpoints are isolated simple eigenvalues: for example, extend the displayed four-dimensional pair by the same scalar $2I$ on an infinite-dimensional summand, and if simple lower endpoints are desired replace the middle block $\diag(3,1)$ of $X$ by $\diag(3,3/2)$ and the middle block of $Y$ by $\left(\begin{smallmatrix}2&1/2\\1/2&2\end{smallmatrix}\right)$. Both middle blocks lie strictly between $I$ and $4I$, so the same extremal lines yield equality.
\end{remark}

\begin{remark}
The alignment phenomenon appearing in the equality case is reminiscent of the algebraic structures associated with K-loops and gyrovector spaces on the positive cone (see \cite{beneduci2014,abe2015,park2015,kim2016}), but we do not claim that our equality condition is equivalent to vanishing gyration in general; this remains an open problem. Our matrix proof for $n\le3$ gives a self-contained verification in low dimensions, and the alignment principle (Lemma \ref{lem:align}) extends to the intermediate infinite-dimensional regime.
\end{remark}

\subsection{Tensor Products and Composite Systems}

\begin{proposition}[Tensor product additivity]\label{prop:tensor-add}
For bounded invertible operators $A,B$ on Hilbert spaces,
\[
\kappa(A\otimes B)=\kappa(A)\kappa(B), \qquad \QLC(A\otimes B)=\QLC(A)+\QLC(B).
\]
\end{proposition}
\begin{proof}
The spatial Hilbert tensor product satisfies $\|A\otimes B\|=\|A\|\|B\|$. The lower bound follows by choosing approximate norm-attaining unit vectors for each factor and taking their tensor product. The upper bound follows by applying the two factors separately. Apply the same argument to $(A\otimes B)^{-1}=A^{-1}\otimes B^{-1}$ and multiply the two norm identities.
\end{proof}

\begin{proposition}[Local vs. global contrast]\label{prop:local-global}
Let $Z\in B(H_1\otimes H_2)$ be invertible. If both factors are finite-dimensional, then
\[
\QLC(Z) \ge \max\{\QLC(\operatorname{Tr}_{H_2}|Z|), \QLC(\operatorname{Tr}_{H_1}|Z|)\}.
\]
If only one factor is finite-dimensional, the same bound holds for the partial trace over that factor.
\end{proposition}
\begin{proof}
Write $Z>0$ (the general case follows by applying the argument to $|Z|$; see Remark \ref{rem:polar} for why this reduction is valid here but not for commutativity statements). The operator inequalities $\lambda_{\min}(Z)I\le Z\le\lambda_{\max}(Z)I$ are preserved under the positive partial trace. Since $\operatorname{Tr}_{H_2}I=d_2I_1$,
\[
d_2\lambda_{\min}(Z)I_1 \le \operatorname{Tr}_{H_2}Z \le d_2\lambda_{\max}(Z)I_1,
\]
so $\lambda_{\min}(\operatorname{Tr}_{H_2}Z)\ge d_2\lambda_{\min}(Z)$ and $\lambda_{\max}(\operatorname{Tr}_{H_2}Z)\le d_2\lambda_{\max}(Z)$. Hence
\[
\kappa(\operatorname{Tr}_{H_2}Z) \le \frac{d_2\lambda_{\max}(Z)}{d_2\lambda_{\min}(Z)} = \kappa(Z),
\]
and taking logarithms gives $\QLC(\operatorname{Tr}_{H_2}Z)\le\QLC(Z)$. The same argument applies to the other partial trace.
\end{proof}

\begin{remark}
No finite upper bound in terms of the two local contrasts alone exists. For $Z_\varepsilon=I+\varepsilon(\sigma_x\otimes\sigma_x+\sigma_y\otimes\sigma_y)$, $0<\varepsilon<1/2$, both partial traces are $2I$. Its eigenvalues are $1$, $1$, and $1\pm2\varepsilon$, so $\QLC(Z_\varepsilon)\to\infty$ as $\varepsilon\uparrow1/2$.
\end{remark}

\begin{proposition}[Chain products]
For a chain $H=\bigotimes_{i=1}^N H_i$ and nearest-neighbour invertible operators $A_{i,i+1}$,
\[
\QLC(A_{12}A_{23}\cdots A_{N-1,N}) \le \sum_{i=1}^{N-1}\QLC(A_{i,i+1}).
\]
Equality is equivalent to saturation of both norm bounds at each composition step.
\end{proposition}
\begin{proof}
Iterate Theorem \ref{thm:subadd} on the prefix products. At every step the forward norm ratio and the reversed inverse norm ratio lie in $(0,1]$, and the product of all the ratios equals one if and only if each ratio equals one. For positive pairs Lemma \ref{lem:align} gives the corresponding criterion in terms of extremal directions.
\end{proof}

\begin{proposition}[Tensor powers]
$\QLC(A^{\otimes N})=N\,\QLC(A)$.
\end{proposition}
\begin{proof}
Immediate from Proposition \ref{prop:tensor-add}.
\end{proof}

We now show that the dimensional threshold of Theorem \ref{thm:threshold} interacts sharply with tensor-product additivity, giving a criterion independent of total ambient dimension. The key technical step is the following elementary but easily mis-stated fact about products of bounded positive reals, which we now prove carefully rather than by assertion.

\begin{lemma}[Tightness of a bounded product]\label{lem:tight-product}
Let $a_1,\dots,a_k,b_1,\dots,b_k$ be positive real numbers with $a_i\le b_i$ for every $i=1,\dots,k$, and suppose
\[
\prod_{i=1}^k a_i = \prod_{i=1}^k b_i.
\]
Then $a_i=b_i$ for every $i=1,\dots,k$.
\end{lemma}
\begin{proof}
We induct on $k$. The case $k=1$ is immediate. For the base case $k=2$: since $a_1\le b_1$, $a_2\le b_2$, and all four numbers are positive,
\[
a_1a_2 \le a_1b_2 \le b_1b_2,
\]
and by hypothesis $a_1a_2=b_1b_2$, so both inequalities are equalities. From $a_1b_2=b_1b_2$ and $b_2>0$ we get $a_1=b_1$. From $a_1a_2=a_1b_2$ and $a_1=b_1>0$ we get $a_2=b_2$.

Now assume the lemma holds for $k-1$ factors, and let $a_1,\dots,a_k,b_1,\dots,b_k>0$ satisfy $a_i\le b_i$ for all $i$ and $\prod_{i=1}^k a_i=\prod_{i=1}^k b_i$. Set $A:=\prod_{i=1}^{k-1}a_i$ and $B:=\prod_{i=1}^{k-1}b_i$. Since each $a_i\le b_i$ and all terms are positive, $A\le B$ (a product of positive numbers is monotone in each factor). The hypothesis reads $Aa_k=Bb_k$, with $A\le B$ and $a_k\le b_k$, both pairs positive; applying the $k=2$ base case to the pair $(A,a_k)$ and $(B,b_k)$ gives $A=B$ and $a_k=b_k$. Having established $A=B$, i.e.\ $\prod_{i=1}^{k-1}a_i=\prod_{i=1}^{k-1}b_i$ with $a_i\le b_i$ for $i=1,\dots,k-1$, the induction hypothesis gives $a_i=b_i$ for $i=1,\dots,k-1$. Together with $a_k=b_k$, this proves the claim for $k$ factors.
\end{proof}

\begin{proposition}[Equality transfer under pure tensor products]\label{prop:tensor-transfer}
Let $X=X_1\otimes\cdots\otimes X_k$ and $Y=Y_1\otimes\cdots\otimes Y_k$, with $X_i,Y_i\in M_{n_i}(\mathbb{C})_{++}$. If equality holds in the subadditivity law,
\[
\QLC(XY) = \QLC(X)+\QLC(Y),
\]
then equality holds factorwise: $\kappa(X_iY_i)=\kappa(X_i)\kappa(Y_i)$ for every $i=1,\dots,k$. In particular, if $n_i\le3$ for every $i$, then $X_iY_i=Y_iX_i$ for every $i$, and consequently $XY=YX$.
\end{proposition}
\begin{proof}
By Proposition \ref{prop:tensor-add}, $\kappa(X)=\prod_i\kappa(X_i)$, $\kappa(Y)=\prod_i\kappa(Y_i)$, and since $XY=\bigotimes_i X_iY_i$, also $\kappa(XY)=\prod_i\kappa(X_iY_i)$. The hypothesis $\kappa(XY)=\kappa(X)\kappa(Y)$ thus reads
\[
\prod_{i=1}^k \kappa(X_iY_i) = \prod_{i=1}^k \big[\kappa(X_i)\kappa(Y_i)\big],
\]
where each factor on the left is bounded above by the corresponding factor on the right, $\kappa(X_iY_i)\le\kappa(X_i)\kappa(Y_i)$, by ordinary submultiplicativity of $\kappa$. Applying Lemma \ref{lem:tight-product} with $a_i=\kappa(X_iY_i)$ and $b_i=\kappa(X_i)\kappa(Y_i)$ gives $\kappa(X_iY_i)=\kappa(X_i)\kappa(Y_i)$ for each $i$. When $n_i\le3$, Theorem \ref{thm:threshold} converts this into $X_iY_i=Y_iX_i$; since factors of a Kronecker product commute exactly when each corresponding pair of tensor legs commutes, $XY=YX$ follows.
\end{proof}

\begin{remark}
Proposition \ref{prop:tensor-transfer} shows that the threshold $n\le3$ in Theorem \ref{thm:threshold} is a statement about irreducibility rather than raw ambient dimension: equality still forces commutativity in arbitrarily large $n$, provided $X,Y$ factor as pure tensor products of blocks of size at most $3$. The noncommuting equality-saturating example of Theorem \ref{thm:threshold} therefore cannot be reproduced, at any total dimension, by padding with further tensor factors of size $\le3$; its common tensor decomposition cannot have all factors of dimension at most three.
\end{remark}

\subsection{Power Laws and Interpolation}
For $A\in B(H)_{++}$ and $\alpha\in\mathbb{R}$,
\[
\QLC(A^\alpha) = |\alpha|\,\QLC(A). \tag{5}
\]
Indeed, $\kappa(A^\alpha)=\kappa(A)^{|\alpha|}$.

For $X,Y\in B(H)_{++}$ and $\alpha\in[0,1]$, define $M_\alpha=X^\alpha Y^{1-\alpha}$. Although $M_\alpha$ need not be positive when $X,Y$ do not commute, it is invertible, and $\QLC(M_\alpha)=\QLC(|M_\alpha|)$ by the polar-decomposition extension. Then from subadditivity,
\[
\QLC(|M_\alpha|) \le \alpha\,\QLC(X) + (1-\alpha)\,\QLC(Y). \tag{6}
\]

\subsection{Contraction and Spectral Range Compression}

The logarithmic contrast admits two distinct types of contraction principles.

\paragraph{Contraction from projective Thompson non-expansiveness.} Let $F:B(H)_{++}\to B(H)_{++}$ be non-expansive w.r.t.\ the projective Thompson metric, $\delta_T([F(A)],[F(B)])\le\delta_T([A],[B])$, and suppose $[F(I)]=[I]$. Then, using (1),
\[
\QLC(F(A)) = \delta_T([I],[F(A)]) = \delta_T([F(I)],[F(A)]) \le \delta_T([I],[A]) = \QLC(A).
\]
Thus OLC is contracted by such maps. Since $\QMC=\tanh(\QLC)$, the same holds for OMC.

\paragraph{Spectral range compression.} Suppose $F$ satisfies $M_{F(A)}\le M_A$, $m_{F(A)}\ge m_A$, where $M_A=\|A\|$, $m_A=\|A^{-1}\|^{-1}$. Then $M_{F(A)}/m_{F(A)}\le M_A/m_A$. Since $\QMC(A)=1-2/(M_A/m_A+1)$, $\QLC(A)=\tfrac12\ln(M_A/m_A)$, and both functions are increasing on $[1,\infty)$, we obtain $\QMC(F(A))\le\QMC(A)$, $\QLC(F(A))\le\QLC(A)$.

Unlike the preceding argument, this conclusion does not use the Thompson metric or any non-expansiveness property of $F$; it follows solely from the compression of the spectral range.

\paragraph{Contrast expansion.} The same spectral reasoning gives the converse criterion: if $M_{G(A)}\ge M_A$, $m_{G(A)}\le m_A$, then $\kappa(G(A))\ge\kappa(A)$, and consequently $\QLC(G(A))\ge\QLC(A)$, $\QMC(G(A))\ge\QMC(A)$.

\subsection{Lower Bound via Deviation from Scalar Operators}
For any $A>0$,
\[
\QLC(A) \ge \tfrac12\ln\Big(1+\Big\|I-\frac{A}{\|A\|}\Big\|\Big). \tag{7}
\]
The proof is elementary using $m=\lambda_{\min}(A)$, $M=\lambda_{\max}(A)$; equality holds iff $A$ is scalar.

\subsection{Boundary Behaviour}
If $A_n>0$ satisfies $\sup\sigma(A_n)/\inf\sigma(A_n)\to\infty$, then $\QLC(A_n)\to\infty$. In particular, if $\inf\sigma(A_n)\to0$ while $\sup\sigma(A_n)$ remains bounded away from zero, $\QLC(A_n)\to\infty$ (consistent with Proposition \ref{prop:singular-cont}). If both extremal values tend to zero at comparable rates, the contrast may remain finite.

\subsection{Sum Inequality and Limiting Behaviour}
For positive invertible operators (see \cite{abed2024}),
\[
\QLC(A+B) \le \max\{\QLC(A),\QLC(B)\}. \tag{8}
\]
By iteration, for $A_1,\dots,A_N>0$,
\[
\QLC\Big(\sum_{k=1}^N A_k\Big) \le \max_{1\le k\le N}\QLC(A_k). \tag{9}
\]
For Ces\`aro means $S_N=\tfrac1N\sum_{k=1}^N A_k$, $\QLC(S_N)\le\max_{1\le k\le N}\QLC(A_k)$, and if $S_N\to A$ in operator norm with $A$ invertible, then
\[
\QLC(A) \le \limsup_{N\to\infty}\QLC(S_N) \le \sup_{k\ge1}\QLC(A_k). \tag{10}
\]
For infinite products $P_N=A_1\cdots A_N$,
\[
\QLC(P_N) \le \sum_{k=1}^N \QLC(A_k), \tag{11}
\]
and if the series converges and $P_N\to P$ in operator norm with $P$ invertible, then
\[
\QLC(P) \le \sum_{k=1}^\infty \QLC(A_k). \tag{12}
\]
Even in the commutative case, equality requires alignment of extremal spectral directions.

\subsection{Trace Inequality for Density Operators (Finite Dimensional)}
Let $\dim H=n<\infty$ and $A>0$ with $\operatorname{Tr}(A)=1$. Let $\lambda_{\min}(A)$ be its smallest eigenvalue. Then
\[
\QLC(A) \le -\tfrac12\ln\lambda_{\min}(A). \tag{13}
\]
\begin{proof}
Let $\lambda_{\max}(A)=\|A\|$. Since $\operatorname{Tr}(A)=1$, $\lambda_{\max}(A)\le1$. Hence $\kappa(A)=\lambda_{\max}(A)/\lambda_{\min}(A)\le1/\lambda_{\min}(A)$. Taking logarithms and dividing by 2 yields (13).
\end{proof}
At fixed $m=\lambda_{\min}(A)$, the exact upper bound is
\[
 \QLC(A)\le \tfrac12\ln\frac{1-(n-1)m}{m},
\]
because the other $n-1$ eigenvalues are each at least $m$. The spectrum $(1-(n-1)m,m,\dots,m)$ attains it. For $n>1$ this also proves that (13) is asymptotically sharp as $m\downarrow0$, while the inequality in (13) is strict for a strictly positive state.

\subsection{Double-Sided Products}
A natural class of operators is given by double-sided products $Z=X^{1/2}YX^{1/2}$, where $X>0$ and $Y$ is invertible. $Z$ is invertible; if $Y=Y^*$, $Z$ is self-adjoint; if $Y>0$, $Z>0$.

The general submultiplicativity of the condition number gives
\[
\kappa(X^{1/2}YX^{1/2}) \le \kappa(X^{1/2})\kappa(Y)\kappa(X^{1/2}) = \kappa(X)\kappa(Y).
\]
Consequently,
\[
\QLC(X^{1/2}YX^{1/2}) \le \QLC(X)+\QLC(Y). \tag{14}
\]
When $Y>0$, setting $T=X^{1/2}Y^{1/2}$ gives $TT^*=X^{1/2}YX^{1/2}=Z$, and since the eigenvalues of $TT^*$ are the squares of the singular values of $T$, $\kappa(Z)=\kappa(T)^2$. Hence $\kappa(Z)=\kappa(X^{1/2}Y^{1/2})^2\le\kappa(X)\kappa(Y)$, recovering (14). Thus positivity of $Y$ is not required for the inequality itself; it only provides the additional factorisation through $Y^{1/2}$.

\subsection{Sequential filters}\label{sec:filters}

In finite dimension let \(K\) be invertible with \(K^*K\leq I\). A successful branch of the operation \(\rho\mapsto K\rho K^*\) occurs with probability \(p_K(\rho)=\Tr(K^*K\rho)\). The extrema over density matrices are the maximal and minimal eigenvalues of \(K^*K\), attained on pure states. If \(R(K)=\max_\rho p_K(\rho)/\min_\rho p_K(\rho)\), then
\[
 R(K)=\kappa(K)^2,\qquad \ln R(K)=4\OLC(K).
\]
The joint success probability of \(Y\) followed by \(X\) is \(p_{XY}(\rho)\). Therefore \(R(XY)\leq R(X)R(Y)\). For positive definite \(X,Y\), Lemma \ref{lem:align} and Theorem \ref{thm:threshold} describe saturation of this bound. Each factor can be rescaled to a contraction without changing \(R\) or the equality condition. For instance, the two matrices in the four-dimensional example, divided by \(4\), give \(R(X/4)=R(Y/4)=16\) and \(R(XY/16)=256\). This is a statement about the dependence of the success probability on the input state. Related applications of the Hilbert projective metric to probabilistic quantum operations are discussed in \cite{reeb2011}.

\section{A Worked Example Class: Density Operators}\label{sec:applications}

Density operators are a natural, well-studied family of positive operators, and they let us make the preceding general theory concrete: OLC coincides with a familiar quantity (the logarithmic ratio of the spectral endpoints of $\rho$), von Neumann entropy gives a second, independently well-understood functional to compare it against, and the qubit/qutrit cases are small enough to compute exactly. None of what follows depends on quantum-mechanical structure specific to density operators (tensor products with a physical subsystem interpretation, measurement, dynamics); the results hold for any positive invertible operator of trace $1$, and we use standard quantum-information terminology (state, entropy, divergence) only because it is the terminology this example family is normally discussed in.

\subsection{Density Matrices}

\begin{proposition}
For a density matrix $\rho$ (trace 1, positive), if $\rho$ is invertible, then
\[
\QLC(\rho) = \tfrac12\ln\frac{\lambda_{\max}(\rho)}{\lambda_{\min}(\rho)}.
\]
It vanishes for the maximally mixed state $\rho=I/d$ and for $d\ge2$ it tends to $+\infty$ as positive definite states approach a pure state. Pure states have $\QLC=+\infty$ by our singular convention.
\end{proposition}
\begin{proof}
Immediate from Definition \ref{def:qmc}.
\end{proof}

\begin{remark}[Color images and the $d=3$ appearance]\label{rem:rgb}
For three positive channel intensities $R,G,B$, the diagonal matrix $M=\operatorname{diag}(R,G,B)$ yields
\[
 \QLC(M)=\tfrac12\ln\frac{\max(R,G,B)}{\min(R,G,B)}.
\]
This gives a scale-invariant spread of the logarithmic intensities in a channel model. Diagonal channel matrices commute and do not illustrate noncommuting saturation. In dimension three, Theorem \ref{thm:threshold} says that equality forces commutativity even for positive definite matrices with off-diagonal entries.
\end{remark}

\subsection{Relation to von Neumann Entropy}

\begin{proposition}[Entropy deficit at fixed dimension]\label{prop:entropy-bound}
For a positive definite density matrix $\rho$ on a $d$-dimensional space, the von Neumann entropy $S(\rho)=-\sum_i\lambda_i\ln\lambda_i$ satisfies
\[
0\le \ln d-S(\rho)\le \ln(d\lambda_{\max}(\rho))\le \ln\kappa(\rho)=2\QLC(\rho). \tag{15}
\]
\end{proposition}
\begin{proof}
The entropy is at most $\ln d$ by concavity and its value at the uniform distribution. Since $\lambda_i\le\lambda_{\max}$ and $\sum_i\lambda_i=1$, averaging $-\ln\lambda_i$ gives $S(\rho)\ge-\ln\lambda_{\max}$. Finally, $\lambda_{\min}\le1/d$, so $\lambda_{\max}=\kappa\lambda_{\min}\le \kappa/d$.
\end{proof}

\begin{proposition}[Exact qubit relation]\label{prop:qubit}
Let $d=2$ and write $x=\QMC(\rho)\in[0,1)$. Then
\[
S(\rho) = \ln2 - \tfrac12\big[(1+x)\ln(1+x)+(1-x)\ln(1-x)\big]. \tag{16}
\]
Equivalently, in terms of $\eta=\QLC(\rho)=\operatorname{arctanh}(x)$,
\[
S(\rho) = \ln2 - \tfrac12\big[(1+\tanh\eta)\ln(1+\tanh\eta)+(1-\tanh\eta)\ln(1-\tanh\eta)\big].
\]
The map $x\mapsto S$ is a strictly decreasing bijection $[0,1)\to(0,\ln2]$, with $S=\ln2$ at $x=0$ (maximally mixed) and $S\to0$ as $x\to1$ (pure states).
\end{proposition}
\begin{proof}
Write the eigenvalues of $\rho$ as $p\ge1-p$, so $\kappa(\rho)=p/(1-p)$ and
\[
x = \QMC(\rho) = \frac{\kappa-1}{\kappa+1} = \frac{p-(1-p)}{p+(1-p)} = 2p-1,
\]
hence $p=\tfrac{1+x}2$, $1-p=\tfrac{1-x}2$. Substituting into $S(\rho)=-p\ln p-(1-p)\ln(1-p)$ and collecting the $\ln2$ terms yields (16). Monotonicity follows since $dS/dx=\tfrac12\ln\tfrac{1-x}{1+x}<0$ for $x\in(0,1)$, and the endpoint values are immediate.
\end{proof}

\begin{figure}[htbp]
\centering
\includegraphics[width=0.62\linewidth]{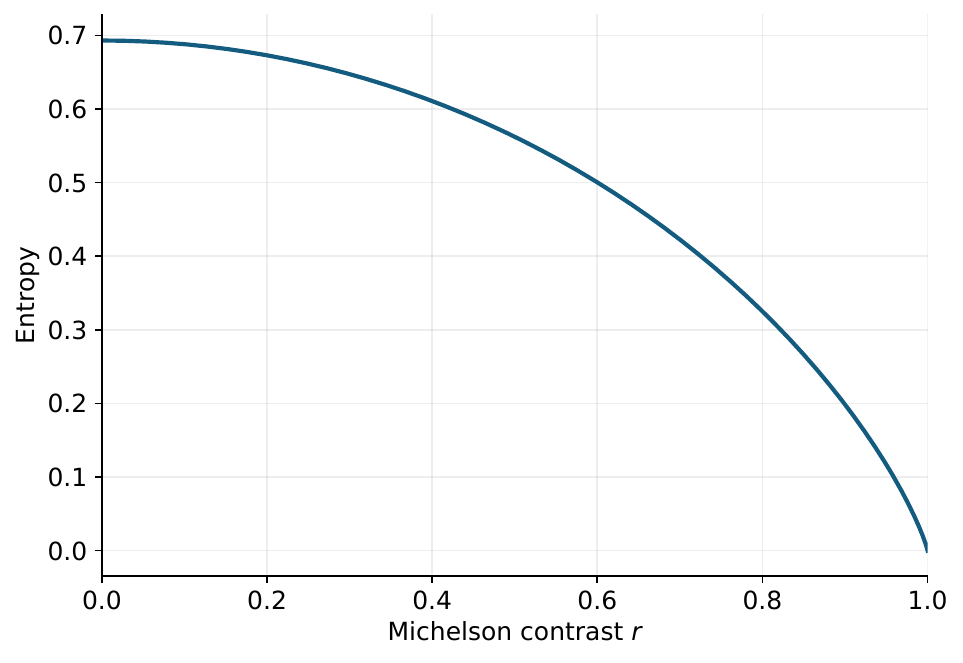}
\caption{$S(\rho)$ as a function of $\QMC(\rho)$ for a qubit density matrix (Proposition \ref{prop:qubit}), linear scale.}
\end{figure}

\begin{figure}[htbp]
\centering
\includegraphics[width=0.62\linewidth]{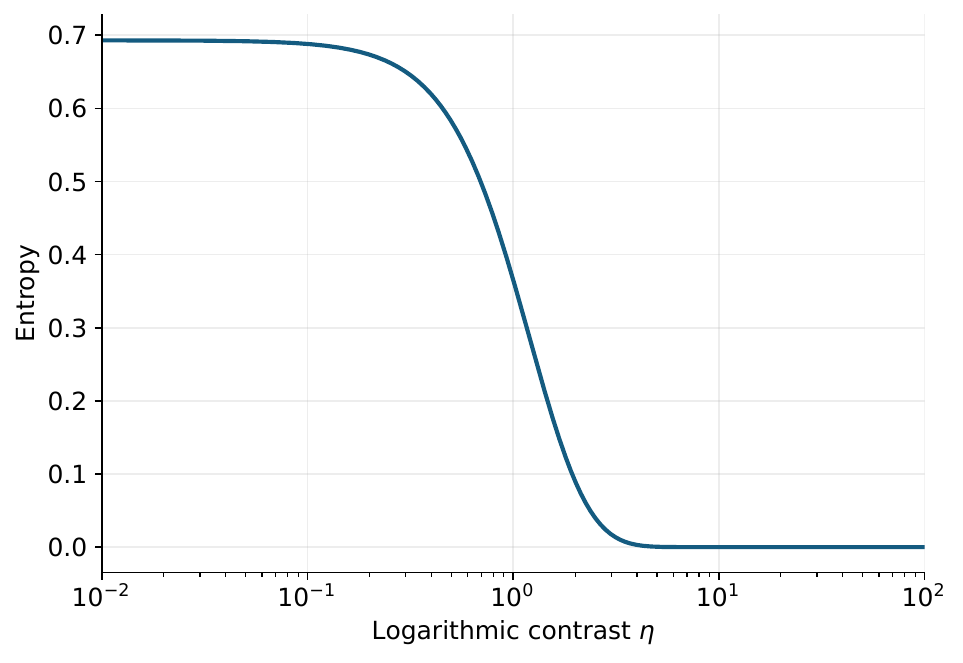}
\caption{$S(\rho)$ as a function of $\QLC(\rho)$ for a qubit density matrix, symmetric log-scale window $\QLC\in[10^{-2},10^2]$ centred on $\QLC=1$.}
\end{figure}

\begin{remark}\label{rem:qubit-remark}
Proposition \ref{prop:qubit} shows that for $d=2$, OMC (equivalently OLC) and $S$ carry exactly the same information, related by the strictly decreasing, strictly concave map (16) -- the binary entropy of the Bloch-sphere radial coordinate $p=(1+x)/2$. Note also that $dS/dx\to0$ as $x\to0$: $S$ is stationary to first order at the maximally mixed state, so OMC (and OLC) are more sensitive local coordinates than $S$ near maximal mixedness. For $d\ge3$ no bijection can hold, since $S$ depends on the full spectrum while OMC depends only on $\lambda_{\max},\lambda_{\min}$; the qutrit bounds and the higher-dimensional family in Section \ref{sec:nonbij} make this failure explicit.
\end{remark}

\begin{remark}[No exact symmetry about $\eta=1$]\label{rem:no-symmetry}
The plot on a logarithmic $\eta$ axis places $\eta$ and $1/\eta$ symmetrically around $1$. If $S(\eta)+S(1/\eta)$ were constant, its limit as $\eta\to\infty$ would give the constant $\ln2$. For large $\eta$, however, $S(1/\eta)=\ln2-1/(2\eta^2)+O(\eta^{-4})$ while $S(\eta)=O(\eta e^{-2\eta})$, by Corollary \ref{cor:entropy-asymp}. The sum is strictly below $\ln2$ for sufficiently large $\eta$.
\end{remark}

\begin{corollary}[Asymptotics of the qubit entropy near maximally mixed and near pure states]\label{cor:entropy-asymp}
Let $\eta=\QLC(\rho)$ for a qubit density matrix $\rho$. Then:
\begin{enumerate}[label=(\alph*)]
\item As $\eta\to0$ (near the maximally mixed state),
\[
S(\eta) = \ln2 - \tfrac12\eta^2 + O(\eta^4);
\]
in particular $S$ is stationary to first order at $\eta=0$ (consistent with Remark \ref{rem:qubit-remark}), and there is no linear dependence of $S$ on $\eta$ near the maximally mixed state.
\item As $\eta\to\infty$ (near a pure state),
\[
S(\eta) = (1+2\eta)e^{-2\eta}\big(1+O(e^{-2\eta})\big),
\]
equivalently
\[
\ln S(\eta) = -2\eta + \ln(1+2\eta) + O(e^{-2\eta}).
\]
Thus $\ln S(\eta)-\ln(1+2\eta)$ is asymptotically linear in $\eta$ with slope $-2$; the difference from its linear asymptote tends to zero at order $e^{-2\eta}$.
\end{enumerate}
Neither of these two asymptotic regimes is linear in $\eta$ globally, and by Remark \ref{rem:no-symmetry} they are not related to one another by any exact symmetry $\eta\leftrightarrow1/\eta$: (a) governs the small-$\eta$ plateau, (b) governs the large-$\eta$ tail, and the transition between the two occurs precisely in the intermediate region $\eta\sim1$ where $\tanh$ leaves its linear regime.
\end{corollary}
\begin{proof}
(a) Write $x=\tanh\eta$ and $f(x):=(1+x)\ln(1+x)+(1-x)\ln(1-x)$, so $S=\ln2-\tfrac12 f(x)$ by (16). Then $f(0)=0$, $f'(x)=\ln\frac{1+x}{1-x}$ so $f'(0)=0$, and $f''(x)=\frac2{1-x^2}$ so $f''(0)=2$; hence $f(x)=x^2+O(x^4)$ and $S=\ln2-\tfrac12x^2+O(x^4)$. Since $x=\tanh\eta=\eta+O(\eta^3)$, $x^2=\eta^2+O(\eta^4)$, giving (a).

(b) Let $\varepsilon:=1-x=1-\tanh\eta$, so $\varepsilon=\dfrac{2e^{-2\eta}}{1+e^{-2\eta}}=2e^{-2\eta}\big(1+O(e^{-2\eta})\big)$. Writing $p=(1+x)/2=1-\varepsilon/2$ and $q=1-p=\varepsilon/2$, the standard small-$q$ expansion of the binary entropy $S=-p\ln p-q\ln q$ gives $S=q(1-\ln q)+O(q^2)$, using $-(1-q)\ln(1-q)=q-\tfrac12q^2+O(q^3)$ together with the exact term $-q\ln q$. Substituting $q=\varepsilon/2$,
\[
S = \frac\varepsilon2\Big(1-\ln\frac\varepsilon2\Big)+O(\varepsilon^2) = \frac\varepsilon2\big(1+\ln(2/\varepsilon)\big)+O(\varepsilon^2).
\]
Substituting $\varepsilon=2e^{-2\eta}(1+O(e^{-2\eta}))$ gives $\varepsilon/2=e^{-2\eta}(1+O(e^{-2\eta}))$ and $\ln(2/\varepsilon)=2\eta+O(e^{-2\eta})$, so
\[
S(\eta) = e^{-2\eta}\big(1+O(e^{-2\eta})\big)\big(1+2\eta+O(e^{-2\eta})\big) = (1+2\eta)e^{-2\eta}\big(1+O(e^{-2\eta})\big),
\]
and taking logarithms gives $\ln S(\eta)=-2\eta+\ln(1+2\eta)+O(e^{-2\eta})$.
\end{proof}

\begin{figure}[htbp]
\centering
\includegraphics[width=0.62\linewidth]{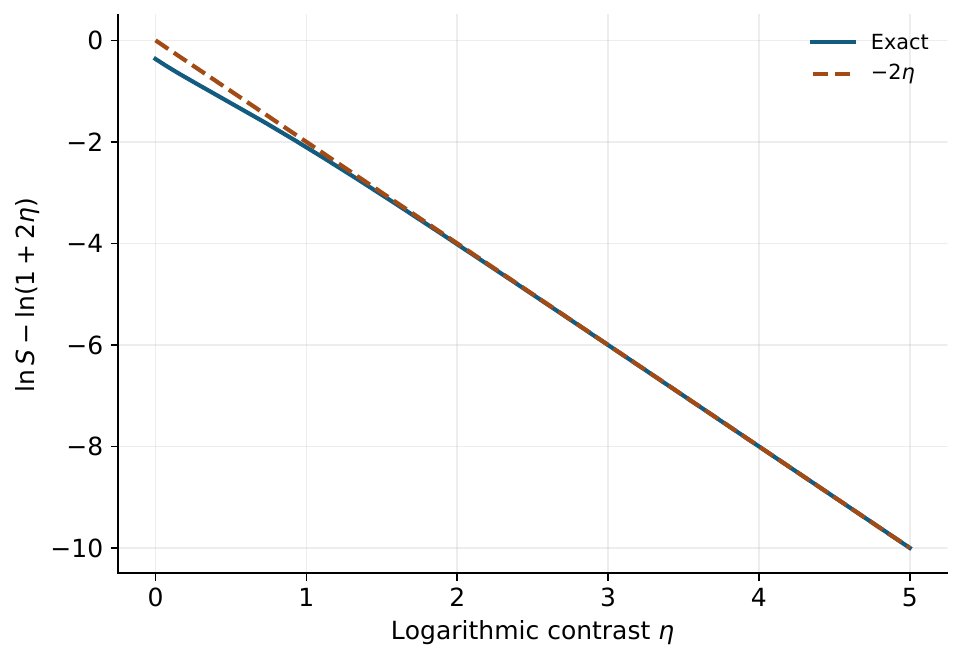}
\caption{The linearised entropy tail of Corollary \ref{cor:entropy-asymp}(b): $\ln S(\eta)-\ln(1+2\eta)$ against $\eta=\QLC(\rho)$, compared with the reference line $-2\eta$. The two curves are visually indistinguishable beyond $\eta\approx2$.}
\end{figure}

\subsection{The Bloch-Sphere Radius and Distance-Based Measures of Mixedness}

For $d=2$, OMC admits a direct geometric meaning as the radial coordinate of the Bloch ball, which in turn makes several standard distance-based measures of mixedness exactly linear in OMC -- in sharp contrast to the entropy, which by Proposition \ref{prop:qubit} is an exact but strictly nonlinear function of OMC.

\begin{proposition}[Bloch-sphere radius identity]\label{prop:bloch}
Write a qubit density matrix as $\rho=\tfrac12(I+\vec r\cdot\vec\sigma)$, where $\vec\sigma=(\sigma_x,\sigma_y,\sigma_z)$ are the Pauli matrices and $r:=|\vec r|\in[0,1)$ is the Bloch-vector length. Then
\[
\QMC(\rho) = r, \qquad \QLC(\rho)=\operatorname{arctanh}(r).
\]
\end{proposition}
\begin{proof}
By the Pauli identity $(\vec a\cdot\vec\sigma)(\vec b\cdot\vec\sigma)=(\vec a\cdot\vec b)I+i(\vec a\times\vec b)\cdot\vec\sigma$ with $\vec a=\vec b=\vec r$, we have $(\vec r\cdot\vec\sigma)^2=r^2I$. Hence the traceless Hermitian operator $\vec r\cdot\vec\sigma$ has eigenvalues $\pm r$, so $\rho$ has eigenvalues $(1\pm r)/2$. Thus $\kappa(\rho)=(1+r)/(1-r)$, and
\[
\QMC(\rho) = \frac{\kappa(\rho)-1}{\kappa(\rho)+1} = \frac{\frac{1+r}{1-r}-1}{\frac{1+r}{1-r}+1} = \frac{(1+r)-(1-r)}{(1+r)+(1-r)} = r,
\]
and $\QLC(\rho)=\operatorname{arctanh}(\QMC(\rho))=\operatorname{arctanh}(r)$.
\end{proof}

\begin{remark}[Extremal cases: pure states and projections]
The Bloch-ball picture gives two extremal cases:
\begin{enumerate}
\item Pure states: $r\to1^-$ gives $\QMC(\rho)\to1$ and $\QLC(\rho)\to\infty$, consistent with the singular convention $\QLC=+\infty$ for singular positive operators.
\item Projections: A rank-1 projection $P=|v\rangle\langle v|$ has eigenvalues $\{1,0,0,\ldots\}$, hence $\kappa(P)=\infty$ and $\QMC(P)=1$, $\QLC(P)=+\infty$. More generally, every nonzero proper orthogonal projection is singular and has these extended values. The identity projection has zero contrast. For the zero operator see the convention in Definition \ref{def:qmc}.
\end{enumerate}
\end{remark}

\begin{corollary}[Linear geometric measures of mixedness]\label{cor:bloch-linear}
With $r=\QMC(\rho)$ as in Proposition \ref{prop:bloch},
\[
D(\rho,I/2) := \tfrac12\|\rho-I/2\|_1 = \frac r2, \qquad
\|\rho-I/2\|_2 = \frac r{\sqrt2}, \qquad
\operatorname{Tr}(\rho^2) = \frac{1+r^2}2.
\]
Thus the trace distance and the Hilbert--Schmidt distance to the maximally mixed state are exactly linear in $\QMC(\rho)$, while the purity is exactly linear in $\QMC(\rho)^2$.
\end{corollary}
\begin{proof}
$\rho-I/2=\tfrac12\vec r\cdot\vec\sigma$ has eigenvalues $\pm r/2$, so its trace norm (sum of absolute eigenvalues) is $r$, giving $D(\rho,I/2)=r/2$; its Hilbert--Schmidt norm is $\sqrt{(r/2)^2+(r/2)^2}=r/\sqrt2$. For the purity, using $(\vec r\cdot\vec\sigma)^2=r^2I$ from Proposition \ref{prop:bloch},
\[
\rho^2=\tfrac14\big(I+2\vec r\cdot\vec\sigma+(\vec r\cdot\vec\sigma)^2\big)=\tfrac14\big((1+r^2)I+2\vec r\cdot\vec\sigma\big),
\]
and taking the trace, using $\operatorname{Tr}(\vec r\cdot\vec\sigma)=0$ and $\operatorname{Tr}I=2$, gives $\operatorname{Tr}(\rho^2)=(1+r^2)/2$.
\end{proof}

\begin{remark}
Proposition \ref{prop:bloch} sharpens the picture of Remark \ref{rem:qubit-remark}: for $d=2$, OMC is not merely in exact bijection with $S$, it is literally the Euclidean radial coordinate of the Bloch ball, while $\QLC=\operatorname{arctanh}(r)$ is the corresponding hyperbolic (Thompson) radial coordinate furnished by Theorem \ref{thm:thompson} applied to the ray $[I]$ representing the maximally mixed state. This makes concrete, for the qubit case, the connection between the geometric identification of Section 3 and the hyperbolic geometry of the Bloch ball studied via Hilbert/Thompson-type projective metrics on qubit density matrices in \cite{park2015}. The linear behaviour of Corollary \ref{cor:bloch-linear} is special to these Euclidean/Hilbert--Schmidt-type distances and to $d=2$; Entropy already has a nonlinear dependence on $r$ in dimension two. In dimension three both entropy and purity vary at fixed contrast, as shown by Propositions \ref{prop:spread} and \ref{prop:purity-spread}. The different convexity of the two functions accounts for the different locations of their extrema.
\end{remark}

\subsection{Composite Systems: A Qutrit--Qubit Example}

Proposition \ref{prop:tensor-transfer} already shows that equality in subadditivity for pure tensor products of factors of dimension $\le3$ always forces commutativity, regardless of total ambient dimension. We record the qutrit--qubit instance explicitly.

\begin{example}[Qutrit $\otimes$ qubit density matrices]
Let $\tau$ be any positive definite qutrit state and put
\[
 \rho=\tau\otimes\frac{\diag(2,1)}{3},\qquad
 \sigma=\tau\otimes\frac{1}{4}
 \begin{pmatrix}2&1\\1&2\end{pmatrix}.
\]
Both are positive definite states on $\mathbb{C}^3\otimes\mathbb{C}^2$. The qubit factors fail to commute, so Theorem \ref{thm:threshold} gives strict subadditivity for their product. By Proposition \ref{prop:tensor-transfer}, the inequality is strict for $\rho,\sigma$ as well. For comparison, replacing the second qubit factor by $\diag(3,1)/4$ aligns its extremal eigenspaces with those of the first factor and gives equality. The common qutrit factor contributes equality in both examples. More generally, equality for any positive product-state pair in this common $3\otimes2$ decomposition forces commutativity of both local pairs.
\end{example}

\subsection{Infinite-Dimensional Perspective: Spectral Measures and Weak Limits}\label{sec:inf-dim}

\label{sec:growing-dimension}
For a positive trace-class operator $\rho$ of trace one on a Hilbert space, entropy can be written via the spectral theorem as $S(\rho)=\sum_{\lambda>0}m_\lambda(-\lambda\ln\lambda)$, with the value $+\infty$ allowed. On an infinite-dimensional space such an operator is compact and cannot be bounded below by a positive multiple of the identity. In particular a finite condition number on its entire space is unavailable, even for a state with zero kernel.

\begin{proposition}[Entropy at bounded condition number in growing dimension]
Let $\rho_j$ be positive definite density matrices of dimensions $d_j\to\infty$, with $\kappa(\rho_j)\le K$ for a finite $K$. Then
\[
 \ln d_j-\ln K\le S(\rho_j)\le\ln d_j,\qquad
 \|\rho_j\|\le K/d_j.
\]
Hence $S(\rho_j)\to\infty$. When the matrices are extended by zero on a common larger Hilbert space, their extensions converge in norm, and thus weakly as operators, to zero.
\end{proposition}
\begin{proof}
These inequalities are (15) and its proof. Zero extension preserves the norm and the trace of each matrix; the norm limit is zero, so the trace is lost in this limiting operation.
\end{proof}

An explicit example at fixed $\kappa=2$ is $\rho_d=\diag(2,1,\ldots,1)/(d+1)$. Its entropy is $\ln(d+1)-2\ln2/(d+1)$, which diverges. The trace-weighted spectral measure of this sequence is supported in $[1/(d+1),2/(d+1)]$ and thus tends weakly to $\delta_0$ on the real line. This convergence gives no bound on entropy, because $-\ln\lambda$ is unbounded near zero. The entropy and purity envelopes below use $d=3$ throughout.

\subsection{\texorpdfstring{Non-bijectivity for $d\ge3$: The Qutrit Spread}{Non-bijectivity for d >= 3: The Qutrit Spread}}\label{sec:nonbij}

For $d=3$, fix the condition number $\kappa=\lambda_1/\lambda_3>1$, where $\lambda_1\ge\lambda_2\ge\lambda_3>0$ are the eigenvalues of $\rho$. Writing $t=\lambda_1$, the trace constraint $\lambda_1+\lambda_2+\lambda_3=1$ together with $\lambda_3=t/\kappa$ forces
\[
\lambda_2(t) = 1-t\Big(1+\frac1\kappa\Big),
\]
and the ordering constraint $\lambda_3\le\lambda_2\le\lambda_1$ restricts $t$ to
\[
t \in \Big[\frac{\kappa}{2\kappa+1},\ \frac{\kappa}{\kappa+2}\Big],
\]
whose lower endpoint is the degeneracy $\lambda_1=\lambda_2$ and whose upper endpoint is $\lambda_2=\lambda_3$. This is a genuine one-parameter family of spectra sharing the same $\kappa$, unlike the qubit case, where $\kappa$ determines the spectrum outright.

\begin{proposition}[Entropy spread at fixed condition number, $d=3$]\label{prop:spread}
Along the family above, $S(t)=-\sum_i\lambda_i(t)\ln\lambda_i(t)$ is strictly concave in $t$. Consequently:
\begin{enumerate}[label=(\alph*)]
\item The maximum $S_{\max}(\kappa)$ is attained at the unique interior point where $\lambda_2/\lambda_1=\kappa^{-1/(\kappa+1)}$, equivalently at
\[
\lambda_1^* = \frac1{1+\kappa^{-1/(\kappa+1)}+1/\kappa}, \qquad \lambda_2^*=\kappa^{-1/(\kappa+1)}\lambda_1^*, \qquad \lambda_3^*=\lambda_1^*/\kappa.
\]
\item The minimum $S_{\min}(\kappa)$ is attained at one of the two boundary spectra,
\[
S_P(\kappa) = \frac{2\kappa}{2\kappa+1}\ln\frac{2\kappa+1}{\kappa} + \frac1{2\kappa+1}\ln(2\kappa+1) \qquad (\lambda_1=\lambda_2),
\]
\[
S_Q(\kappa) = \frac{\kappa}{\kappa+2}\ln\frac{\kappa+2}{\kappa} + \frac2{\kappa+2}\ln(\kappa+2) \qquad (\lambda_2=\lambda_3),
\]
and $S_Q(\kappa) < S_P(\kappa)$ for every $\kappa>1$ (proved analytically in Lemma \ref{lem:SPSQ} below), so $S_{\min}(\kappa)=S_Q(\kappa)$.
\end{enumerate}
Moreover $S_{\min}(\kappa)<S_{\max}(\kappa)$ strictly for every $\kappa>1$, so $\kappa$ (equivalently OMC or OLC) does not determine $S$ for $d=3$: the fiber over each value is a nondegenerate interval, not a point.
\end{proposition}
\begin{proof}
\emph{Concavity.} The map $t\mapsto(\lambda_1,\lambda_2,\lambda_3)(t)$ is affine, and Shannon entropy is strictly concave on the probability simplex; a strictly concave function composed with a non-constant affine map is strictly concave in $t$. A strictly concave function on a closed interval attains its maximum at an interior stationary point (if one exists) and its minimum only at an endpoint.

\emph{Stationary point.} Differentiating $S(t)=-t\ln t-\lambda_2\ln\lambda_2-\lambda_3\ln\lambda_3$ with $\lambda_2'(t)=-(\kappa+1)/\kappa$, $\lambda_3'(t)=1/\kappa$, and setting $dS/dt=0$, the constant terms cancel, leaving
\[
(\kappa+1)\ln\lambda_2 = \kappa\ln\lambda_1+\ln\lambda_3.
\]
Substituting $\lambda_3=\lambda_1/\kappa$ gives $\lambda_2^{\kappa+1}=\lambda_1^{\kappa+1}/\kappa$, i.e.\ $\lambda_2/\lambda_1=\kappa^{-1/(\kappa+1)}$, from which $\lambda_1^*,\lambda_2^*,\lambda_3^*$ follow via the trace constraint.

\emph{Endpoints.} Direct substitution of $t=\kappa/(2\kappa+1)$ and $t=\kappa/(\kappa+2)$ into $S(t)$ gives $S_P,S_Q$ above; part (b) is proved in Lemma \ref{lem:SPSQ}.
\end{proof}

We now supply the analytic proof of $S_Q(\kappa)<S_P(\kappa)$ that Proposition \ref{prop:spread}(b) uses; the original argument for this inequality was only checked numerically, which we regard as insufficient for a paper otherwise proved throughout by closed-form estimates.

\begin{lemma}\label{lem:SPSQ}
For every $\kappa>1$, $S_Q(\kappa) < S_P(\kappa)$.
\end{lemma}
\begin{proof}
First simplify both closed forms. Writing $S_P(\kappa)=\tfrac{2\kappa}{2\kappa+1}[\ln(2\kappa+1)-\ln\kappa]+\tfrac1{2\kappa+1}\ln(2\kappa+1)$ and collecting the $\ln(2\kappa+1)$ terms (their coefficients sum to $\tfrac{2\kappa+1}{2\kappa+1}=1$) gives
\[
S_P(\kappa) = \ln(2\kappa+1) - \frac{2\kappa}{2\kappa+1}\ln\kappa.
\]
An identical manipulation gives
\[
S_Q(\kappa) = \ln(\kappa+2) - \frac{\kappa}{\kappa+2}\ln\kappa.
\]
Define $\varphi(\kappa):=S_P(\kappa)-S_Q(\kappa)$ for $\kappa\ge1$. Then
\[
\varphi(\kappa) = \ln\frac{2\kappa+1}{\kappa+2} - \ln\kappa\left(\frac{2\kappa}{2\kappa+1}-\frac{\kappa}{\kappa+2}\right).
\]
A direct computation gives
\[
\frac{2\kappa}{2\kappa+1}-\frac{\kappa}{\kappa+2} = \kappa\cdot\frac{2(\kappa+2)-(2\kappa+1)}{(2\kappa+1)(\kappa+2)} = \frac{3\kappa}{(2\kappa+1)(\kappa+2)},
\]
so, writing $D(\kappa):=(2\kappa+1)(\kappa+2)=2\kappa^2+5\kappa+2$,
\[
\varphi(\kappa) = \ln\frac{2\kappa+1}{\kappa+2} - \frac{3\kappa\ln\kappa}{D(\kappa)}.
\]
At $\kappa=1$: $\tfrac{2\kappa+1}{\kappa+2}=1$ so the first term vanishes, and $\ln\kappa=0$ kills the second term; hence $\varphi(1)=0$ (consistent with $S_P(1)=S_Q(1)=\ln3$, both computed directly from the simplified forms above).

We show $\varphi'(\kappa)>0$ for $\kappa>1$, which together with $\varphi(1)=0$ gives $\varphi(\kappa)>0$, i.e.\ $S_P(\kappa)>S_Q(\kappa)$, for all $\kappa>1$. Differentiating the first term,
\[
\frac{d}{d\kappa}\ln\frac{2\kappa+1}{\kappa+2} = \frac2{2\kappa+1}-\frac1{\kappa+2} = \frac{2(\kappa+2)-(2\kappa+1)}{D(\kappa)} = \frac{3}{D(\kappa)}.
\]
Differentiating the second term, with $D'(\kappa)=4\kappa+5$,
\[
\frac{d}{d\kappa}\left(\frac{3\kappa\ln\kappa}{D(\kappa)}\right) = \frac{3(\ln\kappa+1)}{D(\kappa)} - \frac{3\kappa\ln\kappa\,D'(\kappa)}{D(\kappa)^2}.
\]
Hence
\begin{align*}
\varphi'(\kappa)
&= \frac{3}{D(\kappa)} - \frac{3(\ln\kappa+1)}{D(\kappa)}
   + \frac{3\kappa\ln\kappa\,D'(\kappa)}{D(\kappa)^2}\\
&= \frac{-3\ln\kappa}{D(\kappa)}
   + \frac{3\kappa(4\kappa+5)\ln\kappa}{D(\kappa)^2}\\
&= \frac{3\ln\kappa\big[\kappa(4\kappa+5)-D(\kappa)\big]}{D(\kappa)^2}.
\end{align*}
Now $\kappa(4\kappa+5)-D(\kappa) = (4\kappa^2+5\kappa) - (2\kappa^2+5\kappa+2) = 2\kappa^2-2 = 2(\kappa-1)(\kappa+1)$, so
\[
\varphi'(\kappa) = \frac{6(\kappa-1)(\kappa+1)\ln\kappa}{D(\kappa)^2}.
\]
For $\kappa>1$, each of $(\kappa-1)$, $(\kappa+1)$, $\ln\kappa$, and $D(\kappa)^2$ is (strictly) positive, so $\varphi'(\kappa)>0$ on $(1,\infty)$. Since $\varphi$ is continuous on $[1,\infty)$ with $\varphi(1)=0$ and $\varphi'>0$ on $(1,\infty)$, we conclude $\varphi(\kappa)>0$, i.e.\ $S_Q(\kappa)<S_P(\kappa)$, for every $\kappa>1$.
\end{proof}

\begin{corollary}\label{cor:smax-bounds}
For qutrits and every $\kappa>1$, $\ln2<S_{\max}(\kappa)<\ln3$. The upper envelope is strictly decreasing and tends to $\ln2$.
\end{corollary}
\begin{proof}
Write $s=\kappa^{\kappa/(\kappa+1)}$ and $w=(\kappa+1)/s$. The maximizing spectrum in Proposition \ref{prop:spread} is proportional to $(\kappa,s,1)$. Since $\kappa\ln\kappa=(\kappa+1)\ln s$, substitution gives $S_{\max}=\ln(1+w)$. Now $(\ln w)'=-\ln\kappa/(\kappa+1)^2$, $w(1)=2$, and $w(\kappa)\to1$ as $\kappa\to\infty$. Thus $S_{\max}$ decreases strictly from $\ln3$ to $\ln2$.
\end{proof}

\begin{remark}
As $\kappa\to1^+$ both envelopes tend to $\ln3$. As $\kappa\to\infty$, $S_{\min}\to0$ and the maximizing spectrum tends to $(1/2,1/2,0)$. The width of the entropy range grows strictly, since
\[
 \frac{(S_{\max}-S_{\min})'}{\ln\kappa}
 =\frac{2}{(\kappa+2)^2}-\frac{w}{(1+w)(\kappa+1)^2}>0.
\]
Indeed $w/(1+w)\le2/3$ and $(\kappa+2)/(\kappa+1)\le3/2$, so the second term is at most $3/(2(\kappa+2)^2)$. The width tends to $\ln2$.
\end{remark}

For each $d\ge3$ and each $\kappa>1$, an explicit family with that condition number consists of the eigenvalues $\kappa t,t$ and $d-2$ copies of $(1-(\kappa+1)t)/(d-2)$, with
\[
 \frac{1}{1+(d-1)\kappa}\le t\le\frac{1}{\kappa+d-1}.
\]
The copies lie between $t$ and $\kappa t$. The interval has positive length, and entropy is strictly concave while purity is strictly convex along the resulting affine family. Neither quantity is constant on a condition-number level set. At $\kappa=1$ the state is uniquely maximally mixed in a fixed dimension.

\begin{figure}[htbp]
\centering
\includegraphics[width=0.72\linewidth]{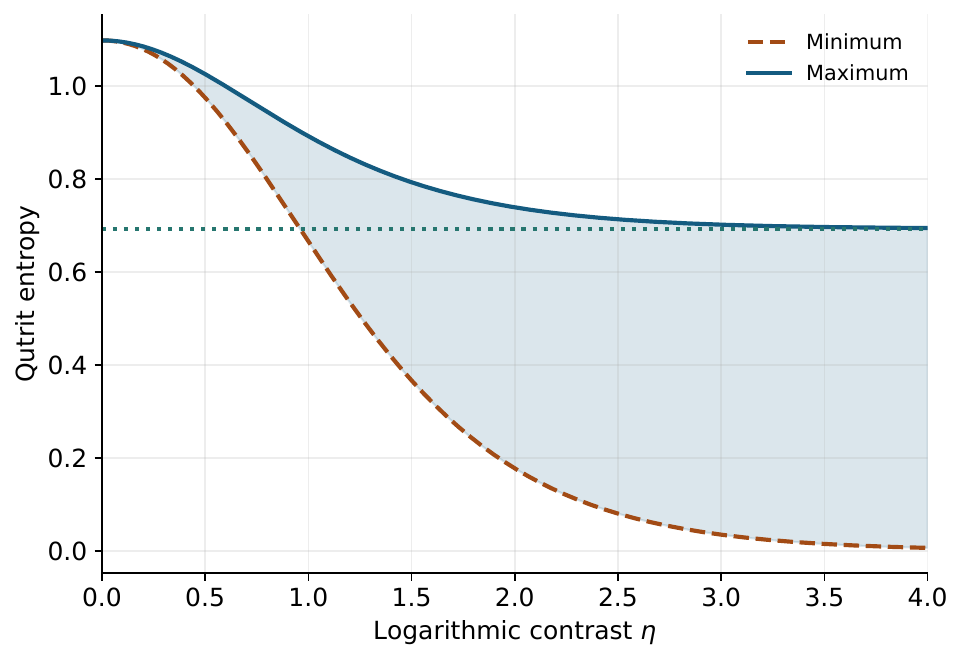}
\caption{Entropy spread $[S_{\min}(\eta),S_{\max}(\eta)]$ for $d=3$ at fixed $\eta=\mathrm{OLC}(\rho)=\tfrac12\ln\kappa$ (Proposition \ref{prop:spread}), $\eta\in[0,4]$, linear scale. $S_{\max}(\eta)\to\ln2$ and $S_{\min}(\eta)\to0$ as $\eta\to\infty$, illustrating Corollary \ref{cor:smax-bounds}.}
\label{fig:entropy-spread}
\end{figure}

\begin{table}[htbp]
\centering
\small
\setlength{\tabcolsep}{3.2pt}
\begin{tabular}{c|ccccccccc}
$\kappa$ & 1 & 4 & 9 & 19 & 50 & 100 & 500 & 1000 & 10000 \\
\hline
$S_{\max}(\kappa)$ & 1.0986 & 0.9743 & 0.8688 & 0.7973 & 0.7426 & 0.7213 & 0.7004 & 0.6971 & 0.6937 \\
$S_{\min}(\kappa)$ & 1.0986 & 0.8676 & 0.6002 & 0.3805 & 0.1897 & 0.1101 & 0.0288 & 0.0158 & 0.0020
\end{tabular}

\caption{Qutrit entropy bounds from Proposition \ref{prop:spread} and Corollary \ref{cor:smax-bounds}.}
\label{tab:entropy}
\end{table}

\subsection{\texorpdfstring{Purity Spread for $d=3$}{Purity Spread for d = 3}}\label{sec:purity-spread}

Corollary \ref{cor:bloch-linear} already gives, for $d=2$, an exact closed-form linear relation between purity and $\QMC(\rho)^2$. For $d=3$ the same one-parameter family of spectra used in Section \ref{sec:nonbij} lets us carry out the analogous analysis for the purity $\operatorname{Tr}(\rho^2)=\lambda_1^2+\lambda_2^2+\lambda_3^2$ at fixed condition number $\kappa$. Unlike the entropy case, this reduces to an honest quadratic in $t$, so every extremal quantity below is a rational function of $\kappa$ -- there is no transcendental equation to solve -- and, since the parabola opens upward rather than downward, the roles of the interior and boundary extrema are exactly reversed relative to Proposition \ref{prop:spread}.

\begin{proposition}[Purity spread at fixed condition number, $d=3$]\label{prop:purity-spread}
Along the family $\lambda_1=t$, $\lambda_2(t)=1-t(1+1/\kappa)$, $\lambda_3=t/\kappa$, $t\in\big[\tfrac{\kappa}{2\kappa+1},\tfrac{\kappa}{\kappa+2}\big]$ of Section \ref{sec:nonbij}, the purity $P(t):=\lambda_1^2+\lambda_2^2+\lambda_3^2$ is a strictly convex quadratic function of $t$. Consequently:
\begin{enumerate}[label=(\alph*)]
\item The minimum is attained at the unique interior stationary point $t_*=\dfrac{\kappa(\kappa+1)}{2(\kappa^2+\kappa+1)}$, which lies strictly inside the interval for every $\kappa>1$, and equals
\[
P_{\min}(\kappa) = \frac{\kappa^2+1}{2(\kappa^2+\kappa+1)}.
\]
\item The maximum is attained at one of the two boundary spectra,
\[
P_P(\kappa) = \frac{2\kappa^2+1}{(2\kappa+1)^2} \quad(\lambda_1=\lambda_2), \qquad
P_Q(\kappa) = \frac{\kappa^2+2}{(\kappa+2)^2} \quad(\lambda_2=\lambda_3),
\]
and $P_Q(\kappa) > P_P(\kappa)$ for every $\kappa>1$, so $P_{\max}(\kappa)=P_Q(\kappa)$.
\end{enumerate}
Moreover $1/3 < P_{\min}(\kappa) < 1/2$ and $1/3 < P_{\max}(\kappa) < 1$ for every $\kappa>1$; at $\kappa=1$ the interval degenerates to $t=1/3$ and $P_{\min}=P_{\max}=1/3$, the purity of the maximally mixed state $I/3$.
\end{proposition}
\begin{proof}
Write $a:=1+1/\kappa$, so $P(t)=t^2+(1-at)^2+(t/\kappa)^2=At^2-2at+1$ with $A:=1+a^2+1/\kappa^2=2(\kappa^2+\kappa+1)/\kappa^2>0$; this is a strictly convex (upward) parabola, so it attains its minimum at the vertex $t_*=a/A=\kappa(\kappa+1)/(2(\kappa^2+\kappa+1))$ and its maximum at one of the two endpoints of any closed subinterval.

\emph{Interior location of $t_*$.} We check $t_*>\kappa/(2\kappa+1)$: cross-multiplying by the positive denominators and dividing by $\kappa>0$, this is equivalent to $(\kappa+1)(2\kappa+1)>2(\kappa^2+\kappa+1)$, i.e.\ $2\kappa^2+3\kappa+1>2\kappa^2+2\kappa+2$, i.e.\ $\kappa>1$. Likewise $t_*<\kappa/(\kappa+2)$ is equivalent, after the same manipulation, to $(\kappa+1)(\kappa+2)<2(\kappa^2+\kappa+1)$, i.e.\ $\kappa^2+3\kappa+2<2\kappa^2+2\kappa+2$, i.e.\ $\kappa(\kappa-1)>0$, again equivalent to $\kappa>1$. So $t_*$ lies strictly inside the interval exactly when $\kappa>1$.

\emph{Minimum value.} For an upward parabola $At^2-2at+1$, the vertex value is $1-a^2/A$. Since $a^2/A=(\kappa+1)^2/(2(\kappa^2+\kappa+1))$,
\[
P_{\min}(\kappa) = 1-\frac{(\kappa+1)^2}{2(\kappa^2+\kappa+1)} = \frac{2(\kappa^2+\kappa+1)-(\kappa+1)^2}{2(\kappa^2+\kappa+1)} = \frac{\kappa^2+1}{2(\kappa^2+\kappa+1)}.
\]

\emph{Boundary values.} At $t=\kappa/(2\kappa+1)$ ($\lambda_1=\lambda_2=t$, $\lambda_3=t/\kappa$), $P=2t^2+t^2/\kappa^2=t^2(2+1/\kappa^2)$, which simplifies to $P_P(\kappa)=(2\kappa^2+1)/(2\kappa+1)^2$. At $t=\kappa/(\kappa+2)$ ($\lambda_1=t$, $\lambda_2=\lambda_3=t/\kappa$), $P=t^2+2t^2/\kappa^2=t^2(1+2/\kappa^2)$, which simplifies to $P_Q(\kappa)=(\kappa^2+2)/(\kappa+2)^2$.

\emph{$P_Q>P_P$ for $\kappa>1$.} Clearing denominators,
\[
P_Q(\kappa)-P_P(\kappa) = \frac{(\kappa^2+2)(2\kappa+1)^2-(2\kappa^2+1)(\kappa+2)^2}{(\kappa+2)^2(2\kappa+1)^2}.
\]
The numerator expands to $2\kappa^4-4\kappa^3+4\kappa-2=2(\kappa^4-2\kappa^3+2\kappa-1)$, and $\kappa^4-2\kappa^3+2\kappa-1$ factors as $(\kappa-1)(\kappa^3-\kappa^2-\kappa+1)=(\kappa-1)(\kappa-1)^2(\kappa+1)=(\kappa-1)^3(\kappa+1)$. Hence
\[
P_Q(\kappa)-P_P(\kappa) = \frac{2(\kappa-1)^3(\kappa+1)}{(\kappa+2)^2(2\kappa+1)^2},
\]
which is strictly positive for $\kappa>1$ (and zero only at $\kappa=1$), proving $P_{\max}(\kappa)=P_Q(\kappa)$.

\emph{Bounds.} $P_Q(\kappa)<1$ iff $\kappa^2+2<(\kappa+2)^2$, i.e.\ $\kappa>-1/2$, always true; $P_Q(\kappa)>1/3$ iff $3(\kappa^2+2)>(\kappa+2)^2$, i.e.\ $2(\kappa-1)^2>0$, true for $\kappa\neq1$. Similarly $P_{\min}(\kappa)<1/2$ iff $\kappa^2+1<\kappa^2+\kappa+1$, always true for $\kappa>0$, and $P_{\min}(\kappa)>1/3$ iff $3(\kappa^2+1)>2(\kappa^2+\kappa+1)$, i.e.\ $(\kappa-1)^2>0$, true for $\kappa\neq1$.
\end{proof}

\begin{remark}
Proposition \ref{prop:purity-spread} mirrors Proposition \ref{prop:spread} exactly, but with every qualitative feature reversed: purity is convex rather than concave along the family, so the extremum at the interior stationary point is a minimum rather than a maximum, and the same endpoint has the opposite extremal role ($P_{\max}=P_Q$ here, versus $S_{\min}=S_Q$ there). No transcendental equation of the form $\lambda_2/\lambda_1=\kappa^{-1/(\kappa+1)}$ appears; every quantity above is a rational function of $\kappa$, and the inequality $P_Q>P_P$ is established by exact polynomial factorisation rather than by a calculus argument, which is possible precisely because purity is quadratic where entropy is logarithmic.
\end{remark}

The purity range width is strictly increasing for $\kappa>1$. To check this, differentiate $P_{\max}-P_{\min}$:
\[
 \frac{d}{d\kappa}(P_{\max}-P_{\min})=
 \frac{\kappa(\kappa-1)(7\kappa^3+9\kappa^2+6\kappa-4)}
 {2(\kappa+2)^3(\kappa^2+\kappa+1)^2}>0.
\]
The width tends to $1/2$ as $\kappa\to\infty$.

\subsection{Fourier-Theoretic Explanation of Non-Bijectivity}\label{sec:log-spectral}

The existence of the entropy spread (Proposition \ref{prop:spread}) and purity spread (Proposition \ref{prop:purity-spread}) for $d=3$ — and their absence for $d=2$ (Proposition \ref{prop:qubit}) — has a natural explanation in terms of log-spectral Fourier transforms. This observation, while not requiring new proofs, provides the structural reason why the bijectivity breaks at $d=3$.

For a positive definite matrix $A=\sum_j\lambda_jP_j$ on a $d$-dimensional space, where the $\lambda_j$ are the distinct eigenvalues, define the normalized log-spectral measure and its Fourier transform by
\[
 \nu_A=\frac1d\sum_j\operatorname{rank}(P_j)\delta_{\ln\lambda_j},
 \qquad
 \Phi_A(t)=\int e^{itx}\,d\nu_A(x)=\frac1d\operatorname{Tr}(A^{it}).
\]
Equivalently the sum runs over eigenvalues counted with multiplicity with weight $1/d$ each. This construction uses the normalized matrix trace. The spectral-diameter identity below uses the spectrum alone and also holds for bounded positive invertible operators on infinite-dimensional spaces.

\begin{proposition}[OLC as spectral diameter; Fourier characterization for qubits]\label{prop:fourier-char}
For any $A\in B(H)_{++}$,
\[
 \QLC(A)=\tfrac12\operatorname{diam}\sigma(\ln A).
\]
In dimension two, write $q=\QLC(A)$ and $c=\tfrac12(\ln\lambda_{\max}+\ln\lambda_{\min})$. Then
\[
 \nu_A=\tfrac12(\delta_{c-q}+\delta_{c+q}),\qquad
 \Phi_A(t)=e^{itc}\cos(tq).
\]
For a trace-one qubit state, $c=-\ln(2\cosh q)$, so its OLC determines the uncentered transform as well.
\end{proposition}
\begin{proof}
The spectral mapping theorem gives the diameter as $\ln M_A-\ln m_A$. In dimension two the two logarithms are exactly $c\pm q$, giving the Fourier formula; centering by $c$ leaves $\cos(tq)$. For trace one, $e^{c+q}+e^{c-q}=1$ and hence $e^c=(2\cosh q)^{-1}$.
\end{proof}

\begin{corollary}[The log-spectral measure contains information beyond OLC]\label{cor:fourier-nonbij}
In dimension $d=2$ a density matrix with given $\kappa$ has a unique spectrum. For every $d\ge3$ and $\kappa>1$, the family of Section \ref{sec:nonbij} has different log-spectral measures with identical diameter. Consequently the transforms differ even at the same OLC.
\end{corollary}
\begin{proof}
The two eigenvalues of a qubit are fixed by their sum and their ratio. The finite-dimensional family varies the eigenvalues while keeping the ratio. In dimension three, for instance, $(\kappa,\kappa,1)/(2\kappa+1)$ and $(\kappa,1,1)/(\kappa+2)$ give distinct two-point measures with endpoint weights $(2/3,1/3)$ and $(1/3,2/3)$, respectively, in decreasing order. The transforms of distinct finite measures are distinct by uniqueness of characteristic functions. The entropy and purity vary by the strict concavity and strict convexity proved in Section \ref{sec:nonbij}.
\end{proof}

\begin{figure}[htbp]
\centering
\includegraphics[width=0.95\linewidth]{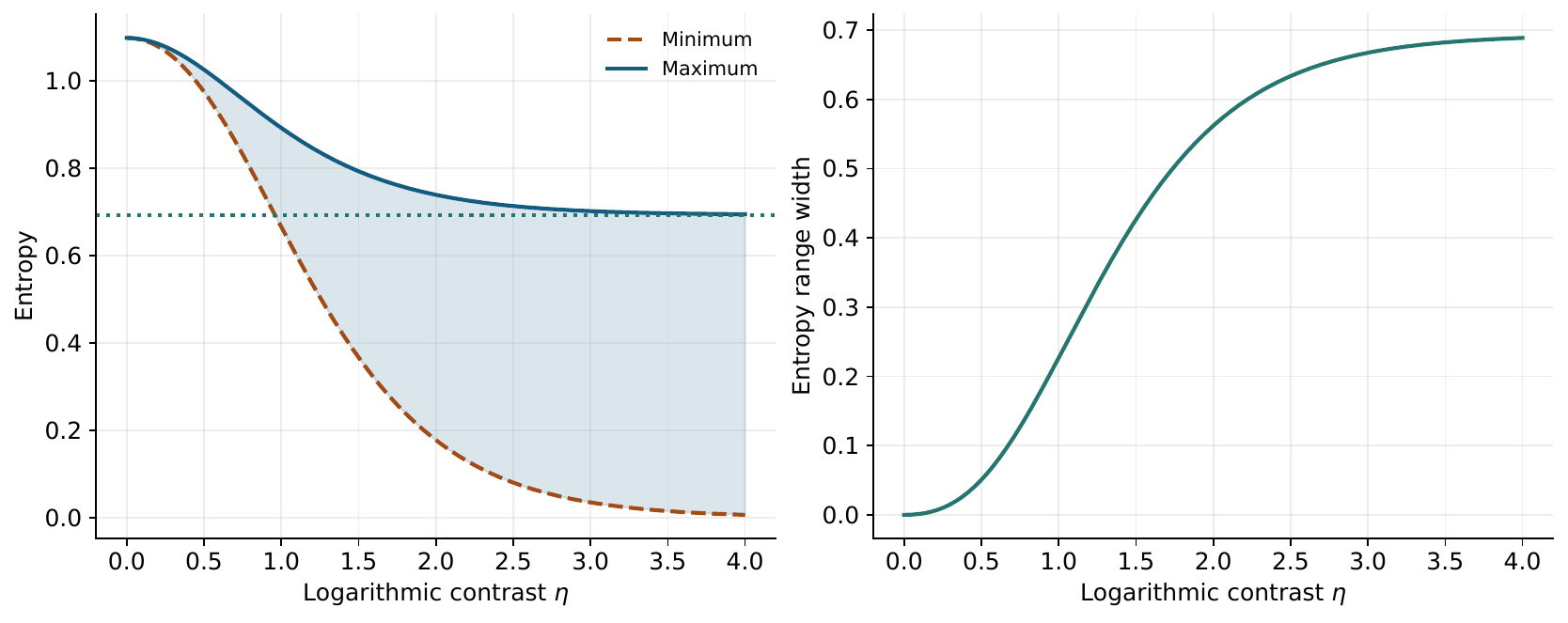}
\caption{Qutrit entropy envelope and its width. Left: the attainable interval at fixed $\eta=\QLC(\rho)=\tfrac12\ln\kappa$, with the reference $\ln2$. Right: the width $S_{\max}-S_{\min}$, which increases to $\ln2$ by Corollary \ref{cor:smax-bounds} and the following remark.}
\label{fig:entropy_envelope}
\end{figure}

\begin{figure}[htbp]
\centering
\includegraphics[width=0.95\linewidth]{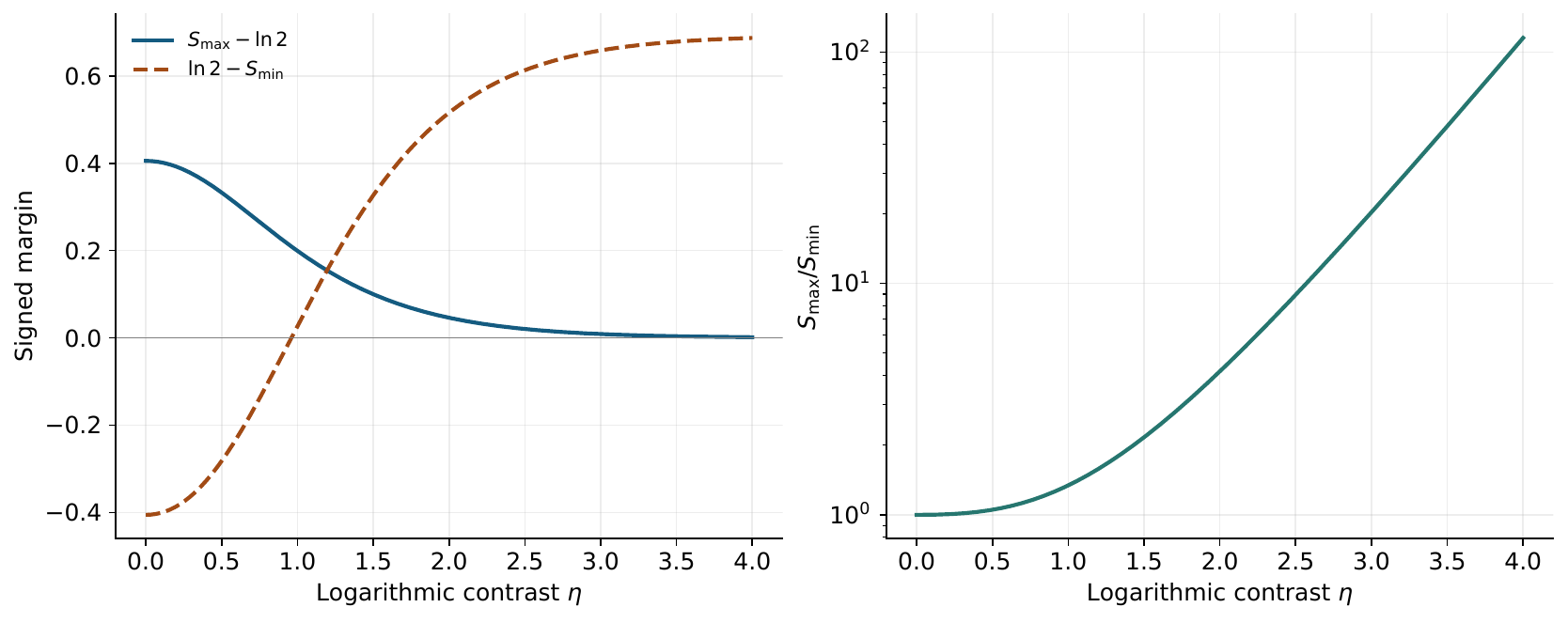}
\caption{Qutrit entropy margins and ratio. Left: $S_{\max}-\ln2$ and the signed margin $\ln2-S_{\min}$. The latter crosses zero when $S_{\min}=\ln2$. Right: $S_{\max}/S_{\min}$, defined for finite $\eta$.}
\label{fig:entropy_margins}
\end{figure}

\begin{figure}[htbp]
\centering
\includegraphics[width=0.95\linewidth]{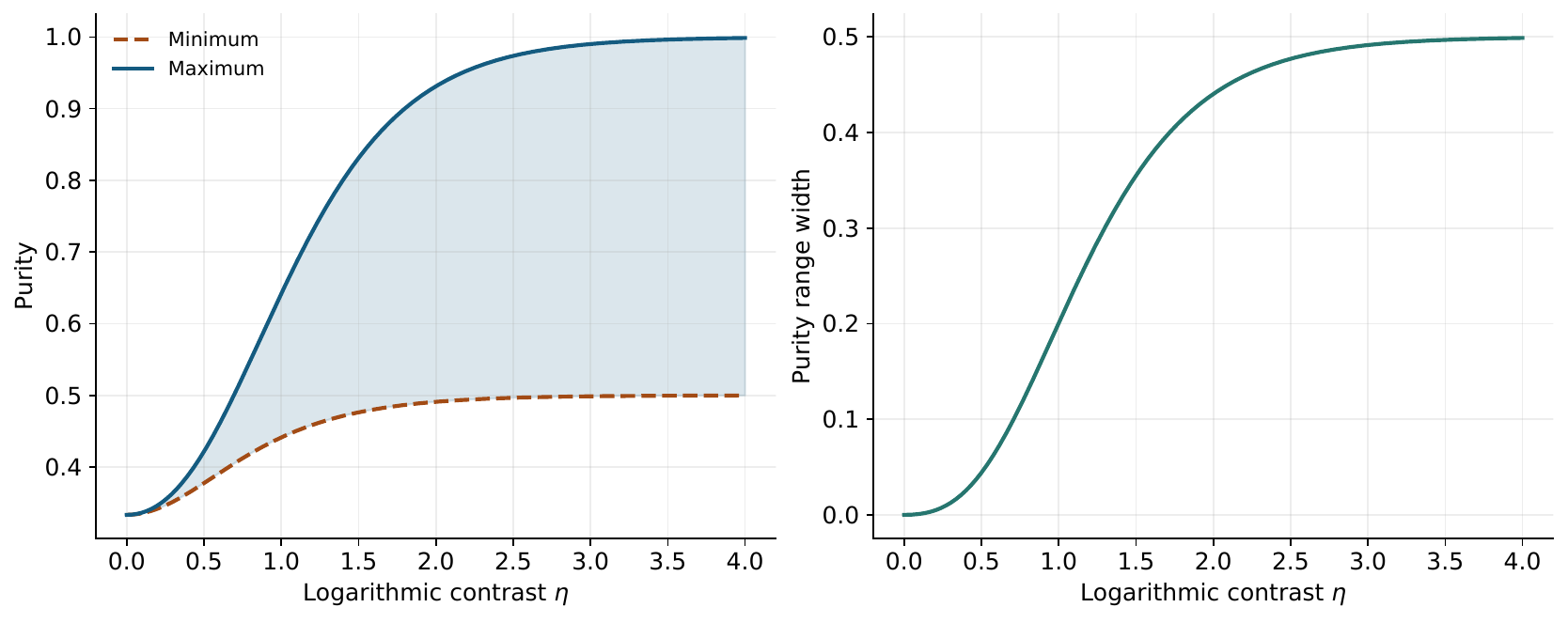}
\caption{Qutrit purity envelope and its width. Left: the attainable interval $[P_{\min},P_{\max}]$ at fixed $\eta$. Right: its width, which increases from $0$ to $1/2$.}
\label{fig:purity_envelope}
\end{figure}

\begin{figure}[htbp]
\centering
\includegraphics[width=0.95\linewidth]{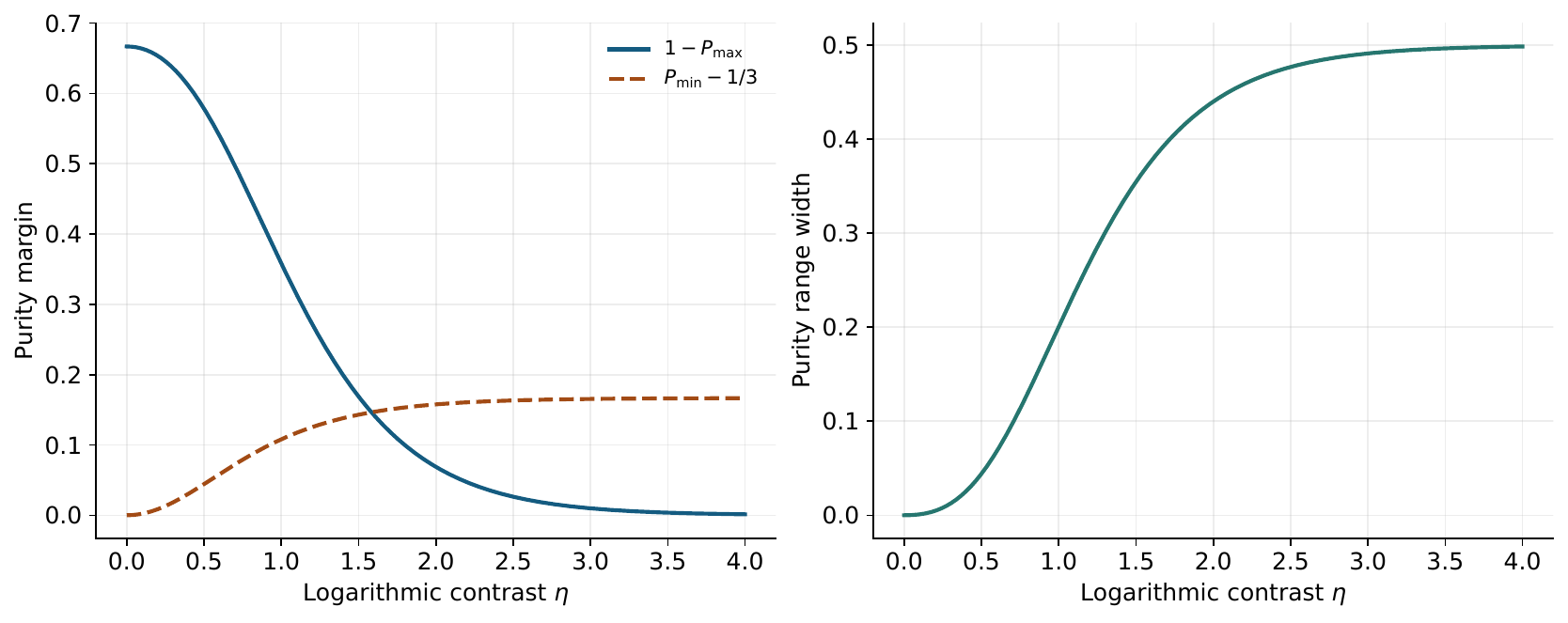}
\caption{Qutrit purity margins and width. Left: distances $1-P_{\max}$ and $P_{\min}-1/3$. The former tends to zero while the latter tends to $1/6$. Right: the width $P_{\max}-P_{\min}$.}
\label{fig:purity_margins}
\end{figure}

\subsubsection{\texorpdfstring{Why $d\ge3$ admits multiple spectra per OLC}{Why d >= 3 admits multiple spectra per OLC}}
At fixed condition number the trace constraint leaves one free spectral coordinate already for a qutrit. Logarithmic contrast uses the separation of the two endpoints but not the relative position or multiplicity of intermediate values. The Fourier transform retains those data. For example, the two measures in the proof of Corollary \ref{cor:fourier-nonbij} are distinguished even after centering the log-spectra by their respective midpoints, because the two endpoint weights are unequal and reversed. This is the spectral information used by the different entropy and purity envelopes in Figures \ref{fig:entropy_envelope} and \ref{fig:purity_envelope}.

\begin{corollary}[Near-pure tail of the purity spread]\label{cor:purity-tail}
As $\kappa\to\infty$ (equivalently $\eta:=\tfrac12\ln\kappa=\QLC\to\infty$),
\[
1-P_{\max}(\eta) = 4e^{-2\eta}\big(1+O(e^{-2\eta})\big), \qquad
\tfrac12-P_{\min}(\eta) = \tfrac12e^{-2\eta}\big(1+O(e^{-2\eta})\big),
\]
equivalently $\ln\big(1-P_{\max}(\eta)\big)=\ln4-2\eta+O(e^{-2\eta})$ and $\ln\big(\tfrac12-P_{\min}(\eta)\big)=-\ln2-2\eta+O(e^{-2\eta})$.
\end{corollary}
\begin{proof}
With $\kappa=e^{2\eta}$, $1-P_Q(\kappa)=\big[(\kappa+2)^2-(\kappa^2+2)\big]/(\kappa+2)^2=(4\kappa+2)/(\kappa+2)^2=(4/\kappa)\big(1+O(1/\kappa)\big)=4e^{-2\eta}\big(1+O(e^{-2\eta})\big)$. Similarly $\tfrac12-P_{\min}(\kappa)=\big[(\kappa^2+\kappa+1)-(\kappa^2+1)\big]/\big(2(\kappa^2+\kappa+1)\big)=\kappa/\big(2(\kappa^2+\kappa+1)\big)=\big(1/(2\kappa)\big)\big(1+O(1/\kappa)\big)=\tfrac12e^{-2\eta}\big(1+O(e^{-2\eta})\big)$. Taking logarithms gives the stated asymptotics.
\end{proof}

\begin{remark}\label{rem:purity-slope}
The slope $-2$ in Corollary \ref{cor:purity-tail} matches the slope $-2$ found for the entropy tail in Corollary \ref{cor:entropy-asymp}(b). This is not a coincidence: in both boundary families, the smallest eigenvalue $\lambda_3=t/\kappa$ scales as $\Theta(1/\kappa)=\Theta(e^{-2\eta})$ near the relevant boundary, and it is this common exponential rate -- not any special feature of entropy or purity individually -- that fixes the slope. We do not attempt to establish this as a general statement for arbitrary spectral functionals $\rho\mapsto-\operatorname{Tr}f(\rho)$ (e.g.\ R\'enyi entropies of general order); we record it as a natural question in Section \ref{sec:open}.
\end{remark}

\subsection{Entanglement Remarks}\label{rem:entanglement-section}

The noncommuting equality example from Theorem \ref{thm:threshold}, after normalisation to unit trace, gives two density matrices $\rho_X,\rho_Y$ with additive OLC but noncommuting. For $\rho_Y=Y/9$ with $Y$ as in the example, the partial transpose over the second qubit yields the matrix
\[
\rho_Y^{T_2} = \frac19\begin{pmatrix}4&0&0&1\\0&2&0&0\\0&0&2&0\\1&0&0&1\end{pmatrix}.
\]
Its eigenvalues are $\tfrac{5+\sqrt{13}}{18},\tfrac29,\tfrac29,\tfrac{5-\sqrt{13}}{18}$, all strictly positive. Hence the PPT criterion is satisfied, which for two-qubit states implies separability \cite{peres1996,horodecki1996}. Thus saturation of subadditivity does not require entanglement.

By Corollary \ref{rem:tensor-irred} below, we can say more precisely where a genuinely noncommuting saturating example must live: it cannot be a pure tensor product of factors of dimension $\le3$, since Proposition \ref{prop:tensor-transfer} rules those out regardless of total dimension.

\begin{corollary}\label{rem:tensor-irred}
Any noncommuting equality-saturating pair of density matrices $\rho,\sigma$ -- entangled or not -- cannot simultaneously be written in a common tensor decomposition as $\rho=\rho_1\otimes\cdots\otimes\rho_k$, $\sigma=\sigma_1\otimes\cdots\otimes\sigma_k$ with every factor of dimension $\le3$.
\end{corollary}
\begin{proof}
Immediate from Proposition \ref{prop:tensor-transfer}: any such decomposition with all $n_i\le3$ would force $\rho_i\sigma_i=\sigma_i\rho_i$ for every $i$, hence $\rho\sigma=\sigma\rho$.
\end{proof}

\begin{remark}\label{rem:entanglement}
As noted in Remark \ref{rem:polar}, it is tempting to try to extend Theorem \ref{thm:threshold} to general (not necessarily positive) invertible $X,Y$ by passing to absolute values $|X|,|Y|$, since $\kappa(X)=\kappa(|X|)$ and $\QLC(X)=\QLC(|X|)$. This is valid for statements purely about $\kappa$ or OLC, but not for the commutativity conclusion itself: $|X|$ and $|Y|$ commuting says nothing about whether $X$ and $Y$ commute (Pauli matrices $X=\sigma_x$, $Y=\sigma_y$ give $|X|=|Y|=I$ trivially commuting while $XY=-YX$). Theorem \ref{thm:threshold} genuinely requires $X,Y$ positive from the outset, since Lemma \ref{lem:align} uses the eigenspaces $\Emax,\Emin$ of $X,Y$ themselves. An entangled equality pair can also be constructed explicitly, as in the next example.
\end{remark}

\begin{example}[An entangled two-qubit equality pair]\label{ex:entangled-saturation}
Use any Bell basis of $\mathbb C^2\otimes\mathbb C^2$, with first vector $\lvert\Phi^+\rangle$. Set
\[
 X=\diag(8,3,2,1),\qquad
 Y=\begin{pmatrix}8&0&0&0\\0&5/2&1/2&0\\0&1/2&5/2&0\\0&0&0&1\end{pmatrix}
\]
in that basis, and normalize both by their trace $14$. Both first and last basis vectors are common extremal eigenvectors. Hence the pair realizes equality by Lemma \ref{lem:align}. The middle blocks do not commute. Both states have overlap $8/14>1/2$ with $\lvert\Phi^+\rangle$. A product unit vector has squared overlap at most $1/2$ with that Bell vector by Cauchy--Schwarz, and the same bound holds for any separable state by convexity. Thus both states are entangled.
\end{example}

\subsection[The Holevo Quantity and the Quantum Jensen--Shannon Divergence]{The Holevo Quantity and the Quantum\texorpdfstring{\protect\newline}{} Jensen--Shannon Divergence}\label{sec:qjsd}

For an ensemble $\{(p_i,\rho_i)\}_{i=1}^N$ of quantum states with average state $\bar\rho=\sum_i p_i\rho_i$, the Holevo $\chi$-quantity $\chi=S(\bar\rho)-\sum_i p_iS(\rho_i)$ upper-bounds the classically accessible information about the ensemble label under any measurement. For the equiprobable binary case $\{(\tfrac12,\rho_1),(\tfrac12,\rho_2)\}$, $\chi$ coincides with the quantum Jensen--Shannon divergence,
\[
\operatorname{QJSD}(\rho_1,\rho_2) := S(\bar\rho) - \tfrac12\big[S(\rho_1)+S(\rho_2)\big], \qquad \bar\rho=\tfrac12(\rho_1+\rho_2),
\]
a quantity already flagged, but not developed, as a keyword of our earlier paper \cite{abed2024}. We record here the exact qubit formula this connection yields, together with a genuinely dimension-free structural fact about $\bar\rho$; we also record, as an explicit warning to a future reader, why the most natural attempt to extend the qubit formula into a nontrivial higher-dimensional \emph{bound} on QJSD fails.

\begin{figure}[htbp]
\centering
\includegraphics[width=0.95\linewidth]{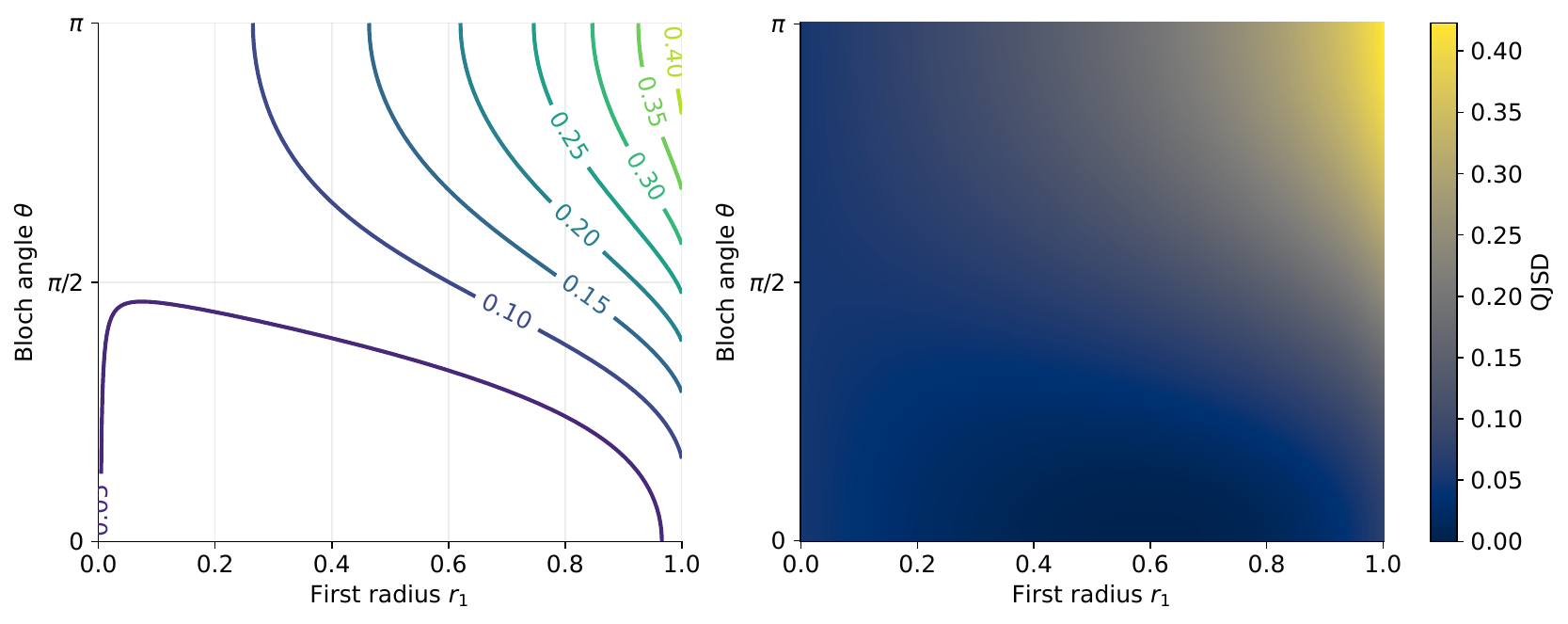}
\caption{Qubit Jensen--Shannon divergence at $r_2=0.6$. Both panels show the same function of $r_1$ and Bloch angle $\theta$. Left: labelled contours. Right: a heat map.}
\label{fig:qjsd_contours}
\end{figure}

\begin{figure}[htbp]
\centering
\includegraphics[width=0.95\linewidth]{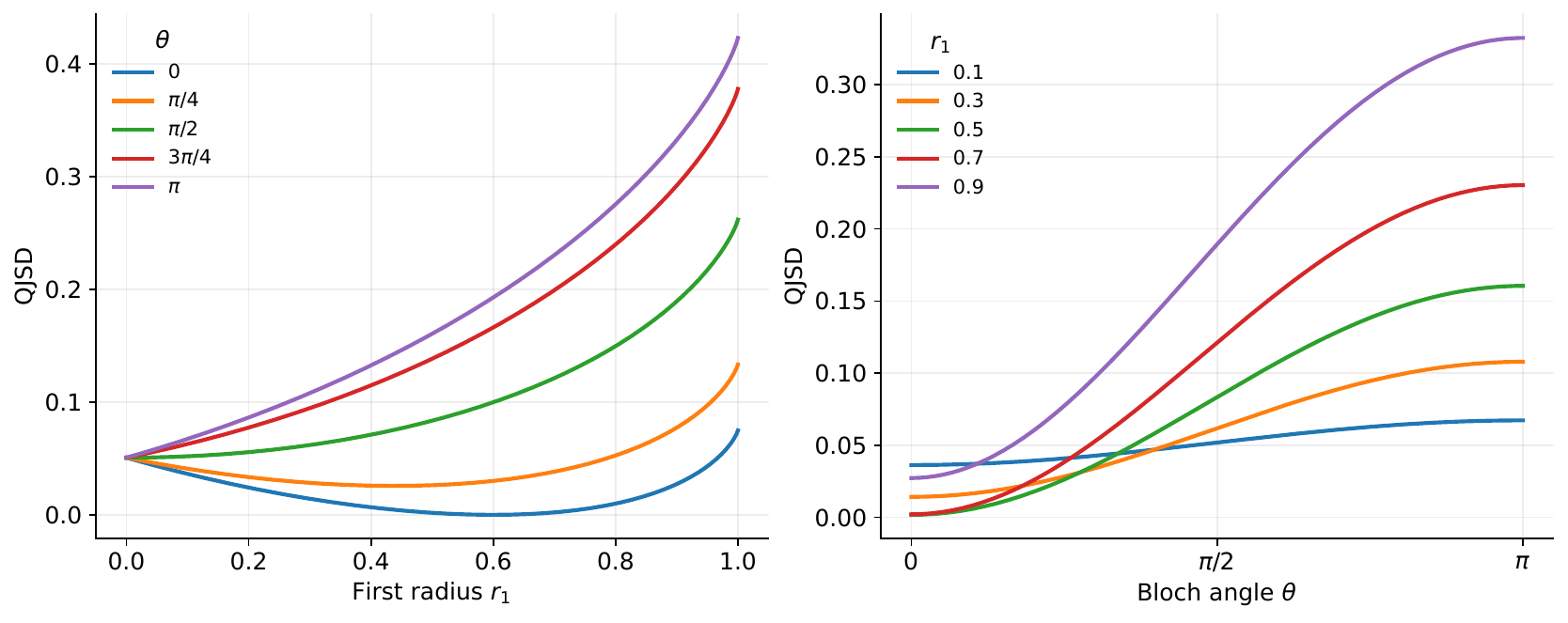}
\caption{Qubit Jensen--Shannon divergence at $r_2=0.6$. Left: first-radius slices at five fixed angles. Right: angle slices at five fixed first radii. The parallel value can be positive when the radii differ. For fixed nonzero radii divergence increases with the Bloch angle.}
\label{fig:qjsd_slices}
\end{figure}

\begin{figure}[htbp]
\centering
\includegraphics[width=0.95\linewidth]{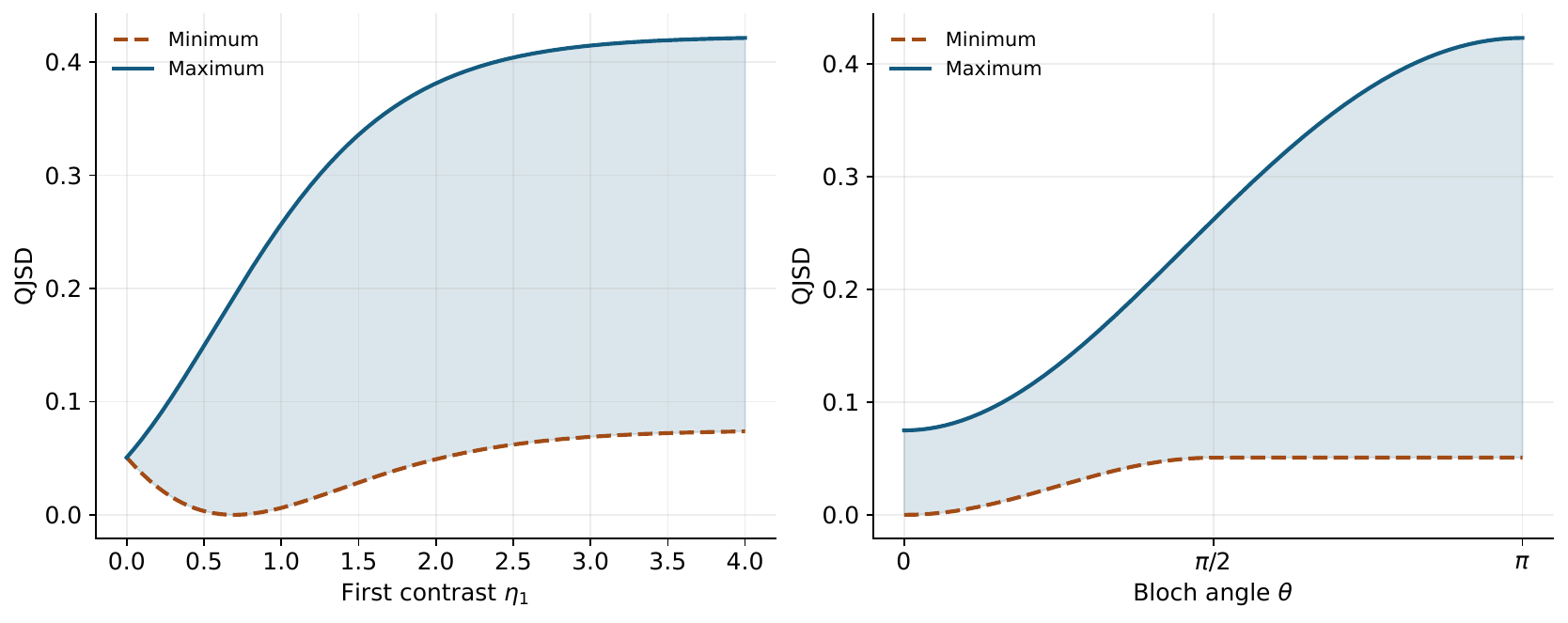}
\caption{Range of the qubit divergence at $r_2=0.6$. Left: minimum and maximum over the angle for each $\eta_1=\QLC(\rho_1)$. Right: numerical extrema over $r_1\in[0,1]$ at each angle. The fixed mixed second state precludes attaining the binary bound $\ln2$ in this plot.}
\label{fig:qjsd_ridges}
\end{figure}

\begin{proposition}[Exact qubit QJSD formula]\label{prop:qjsd-qubit}
Let $\rho_1,\rho_2$ be qubit density matrices with Bloch vectors $\vec r_1,\vec r_2$ of lengths $r_1,r_2\in[0,1)$ and angle $\theta$ between them, and write $H(x):=\ln2-\tfrac12\big[(1+x)\ln(1+x)+(1-x)\ln(1-x)\big]$ for the qubit entropy-vs-OMC function of Proposition \ref{prop:qubit} (extended by $H(1):=0$). Then
\begin{align*}
\bar r &:= \QMC(\bar\rho) = \tfrac12\sqrt{r_1^2+r_2^2+2r_1r_2\cos\theta},\\
\operatorname{QJSD}(\rho_1,\rho_2) &= H(\bar r) - \tfrac12\big[H(r_1)+H(r_2)\big].
\end{align*}
\end{proposition}
\begin{proof}
The Bloch vector of $\bar\rho=\tfrac12(\rho_1+\rho_2)$ is $\tfrac12(\vec r_1+\vec r_2)$, so
\[
\bar r=\big|\tfrac12(\vec r_1+\vec r_2)\big|=\tfrac12\sqrt{|\vec r_1|^2+|\vec r_2|^2+2\vec r_1\cdot\vec r_2}=\tfrac12\sqrt{r_1^2+r_2^2+2r_1r_2\cos\theta}.
\]
By Proposition \ref{prop:qubit}, $S(\rho)=H(\QMC(\rho))$ for any qubit density matrix, and by Proposition \ref{prop:bloch}, $\QMC(\rho_1)=r_1$, $\QMC(\rho_2)=r_2$, $\QMC(\bar\rho)=\bar r$; substituting into the definition of QJSD gives the stated formula.
\end{proof}

\begin{remark}[Consistency checks]
Two extreme cases recover known values. If $\theta=0$ (parallel Bloch vectors, giving commuting states), $\bar r=(r_1+r_2)/2$ and the formula reduces to the classical Jensen--Shannon divergence between the two Bernoulli distributions $(1\pm r_i)/2$, and it vanishes exactly when $r_1=r_2$. The antiparallel case $\theta=\pi$ also gives commuting states, with reversed ordering of eigenvalues in the common basis. For fixed positive radii QJSD increases strictly with $\theta$, since $\bar r$ decreases and $H$ is decreasing. If $\theta=\pi$ and $r_1=r_2=1$ (orthogonal pure states, e.g.\ $|0\rangle\langle0|$ and $|1\rangle\langle1|$), then $\bar r=0$ and $\operatorname{QJSD}=H(0)-\tfrac12[H(1)+H(1)]=\ln2-0=\ln2$, the known maximal value of QJSD for perfectly distinguishable states.
\end{remark}

\begin{remark}[Relation to the classical Michelson--JSD equivalence, and why it does not lift past $d=2$]\label{rem:bruni}
Bruni, Rossi, and Vitulano \cite{bruni2012} prove a formal equivalence, for classical two-level (binary) distributions, between the Michelson contrast of an image distortion and the Jensen--Shannon divergence it induces. Proposition \ref{prop:qjsd-qubit} is exactly the operator lift of that equivalence to $d=2$: $r_i=\QMC(\rho_i)$ plays the role of the Michelson contrast, and QJSD is recovered as a closed-form function of $(r_1,r_2,\theta)$ alone. For $d=2$ this is a genuine, checkable correspondence, not a cosmetic analogy.

It does not extend to $d\ge3$, and the reason is exactly Proposition \ref{prop:spread}: for $d\ge3$, OLC (hence OMC) no longer determines the spectrum, so it cannot determine QJSD either. Concretely, fix $\kappa_1=5$ and take the two extremal $d=3$ spectra of Proposition \ref{prop:spread} sharing that condition number, $\lambda^{\min\text{-entropy}}=(5,1,1)/7$ and $\lambda^{\max\text{-entropy}}\approx(0.508977,0.389227,0.101795)$ (both diagonal in the same basis, so no angle ambiguity enters); pair each against a fixed third state $\rho_2$ with spectrum $(3,1.7,1)/5.7$. Both choices of $\rho_1$ have identical $\QMC(\rho_1)$, identical $\QMC(\rho_2)$, and identical (trivial, commuting) relative orientation with $\rho_2$ -- yet direct computation gives $\operatorname{QJSD}=0.0220$ for the min-entropy spectrum against $0.0080$ for the max-entropy spectrum, a $\sim\!170\%$ relative difference. So already the most favorable case for a closed-form generalization -- commuting states, where there is no orientation angle to worry about -- rules one out: no function of $(\QMC(\rho_1),\QMC(\rho_2),\theta)$ alone can equal QJSD once $d\ge3$. The loss of spectral information is present even when the states share an ordered diagonal basis. For an exact example in every $d\ge3$, let $A=\diag(\kappa,1,\ldots,1)/(\kappa+d-1)$ and $B=\diag(\kappa,\ldots,\kappa,1)/((d-1)\kappa+1)$ for $\kappa>1$. The pairs $(A,A)$ and $(A,B)$ have identical input contrasts and share the displayed ordered diagonal basis. Their divergences are zero and strictly positive, respectively, by strict concavity of entropy.
\end{remark}

\begin{figure}[htbp]
\centering
\includegraphics[width=0.95\linewidth]{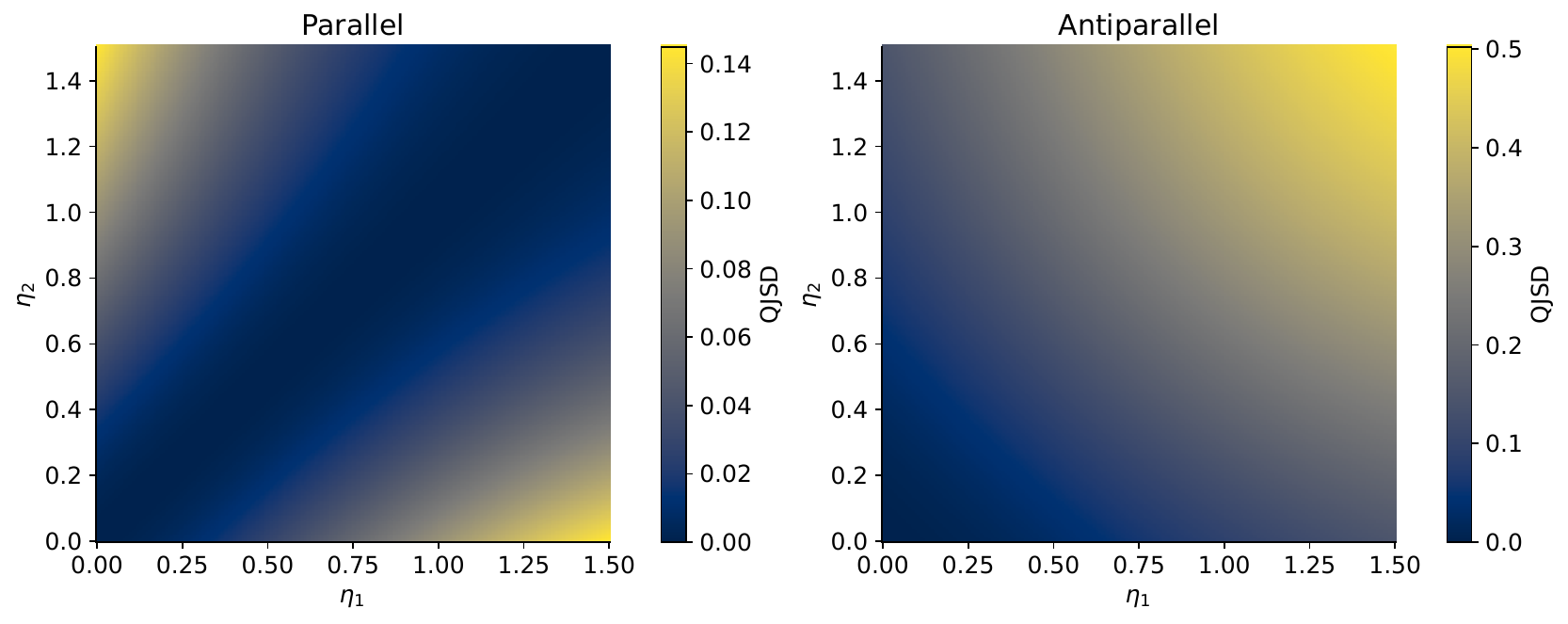}
\caption{Qubit divergence in the two logarithmic contrasts for parallel and antiparallel Bloch vectors. Both views use $\eta_1,\eta_2\in[0,1.5]$. Each color bar is scaled to its own panel. In the $\theta=\pi$ case the two states have orthogonal supports only when both are pure.}
\label{fig:qjsd-contours}
\end{figure}

\begin{figure}[htbp]
\centering
\includegraphics[width=0.95\linewidth]{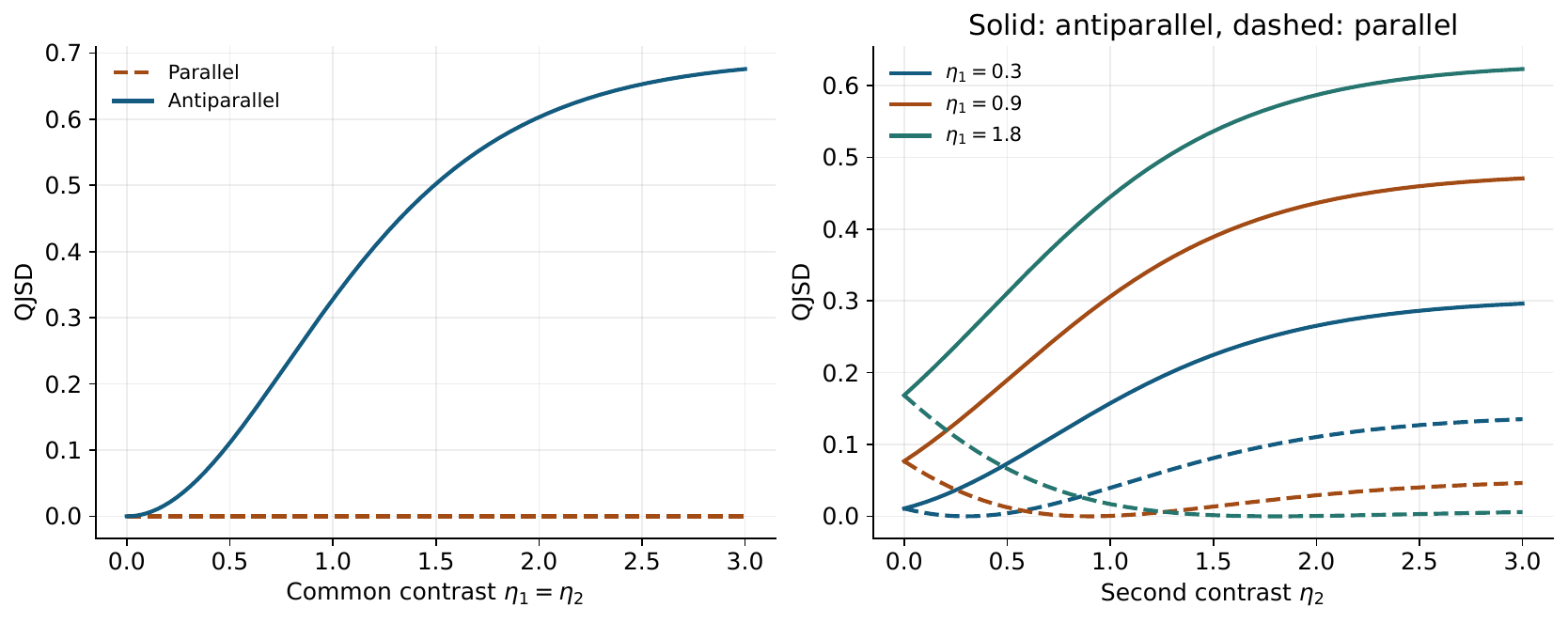}
\caption{Qubit divergence along the diagonal and selected cross-sections. Left: the two orientations at $\eta_1=\eta_2$. The parallel value vanishes along this diagonal. Right: cross-sections at $\eta_1=0.3,0.9,1.8$, showing both orientations. The parallel slice has a minimum at $\eta_2=\eta_1$.}
\label{fig:qjsd-slices}
\end{figure}

\begin{figure}[htbp]
\centering
\includegraphics[width=0.95\linewidth]{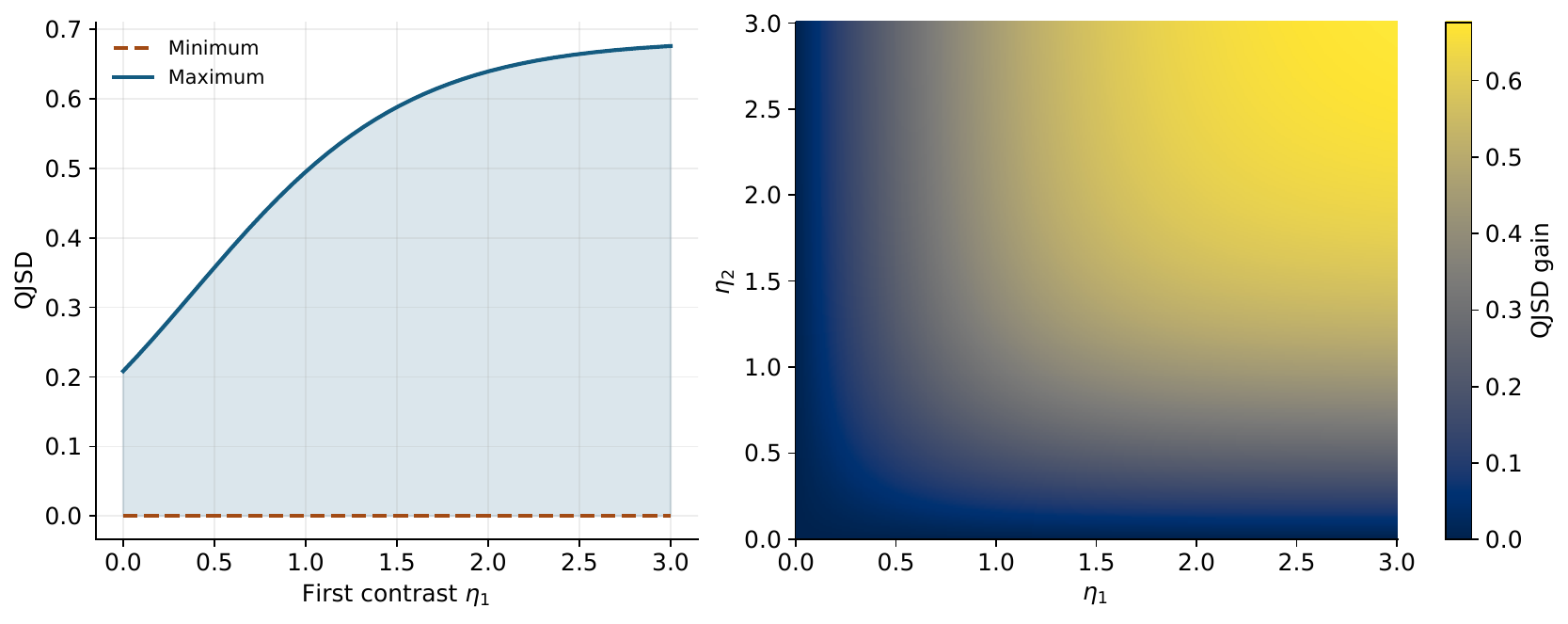}
\caption{Qubit divergence marginals and angular gain. Left: extrema over $\eta_2\in[0,3]$ and $\theta\in[0,\pi]$ for each $\eta_1\in[0,3]$. Right: the gain $\operatorname{QJSD}(\theta=\pi)-\operatorname{QJSD}(\theta=0)$ on the same two-contrast window.}
\label{fig:qjsd-marginals}
\end{figure}

\begin{proposition}[Conditioning under ensemble averaging]\label{prop:ensemble-contraction}
Let $\{(p_i,\rho_i)\}_{i=1}^N$ be a finite ensemble of positive invertible operators $\rho_i\in B(H)_{++}$ with weights $p_i>0$ summing to one, and $\bar\rho:=\sum_{i=1}^N p_i\rho_i$. Then
\[
\QLC(\bar\rho) \le \max_{1\le i\le N} \QLC(\rho_i).
\]
\end{proposition}
\begin{proof}
Apply the sum inequality (9) to $A_i:=p_i\rho_i>0$: $\QLC\big(\sum_iA_i\big)\le\max_i\QLC(A_i)$. Since OLC is invariant under positive scaling (Section 2.2), $\QLC(A_i)=\QLC(p_i\rho_i)=\QLC(\rho_i)$, and $\sum_iA_i=\bar\rho$.
\end{proof}

\begin{remark}[A dimension-dependent bound that is vacuous]\label{rem:qjsd-vacuous}
A bound on $S(\bar\rho)$ alone of the form $\ln d+2\max_i\QLC(\rho_i)$ is at least as large as $\ln d$. It follows already from $S(\bar\rho)\le\ln d$ and nonnegativity of contrast. In particular, it cannot improve the binary bound $\operatorname{QJSD}\le\ln2$ for $d\ge2$. Figure \ref{fig:vacuous-gap} shows the gap. It is positive except at $d=2$ and zero contrast. In (15) the information provided by contrast is a lower bound on entropy, which can instead be used on the two input states as follows.
\end{remark}

\begin{corollary}[A contrast bound near the maximally mixed state]
For positive definite density matrices $\rho_1,\rho_2$ on the same finite-dimensional space,
\[
 \operatorname{QJSD}(\rho_1,\rho_2)
 \le \min\{\ln2,\QLC(\rho_1)+\QLC(\rho_2)\}.
\]
\end{corollary}
\begin{proof}
Write the divergence as $S(\bar\rho)-\tfrac12(S(\rho_1)+S(\rho_2))$. Use $S(\bar\rho)\le\ln d$ and $S(\rho_i)\ge\ln d-2\QLC(\rho_i)$ from (15), then combine with the binary Holevo bound $\ln2$.
\end{proof}

\begin{figure}[htbp]
\centering
\includegraphics[width=0.95\linewidth]{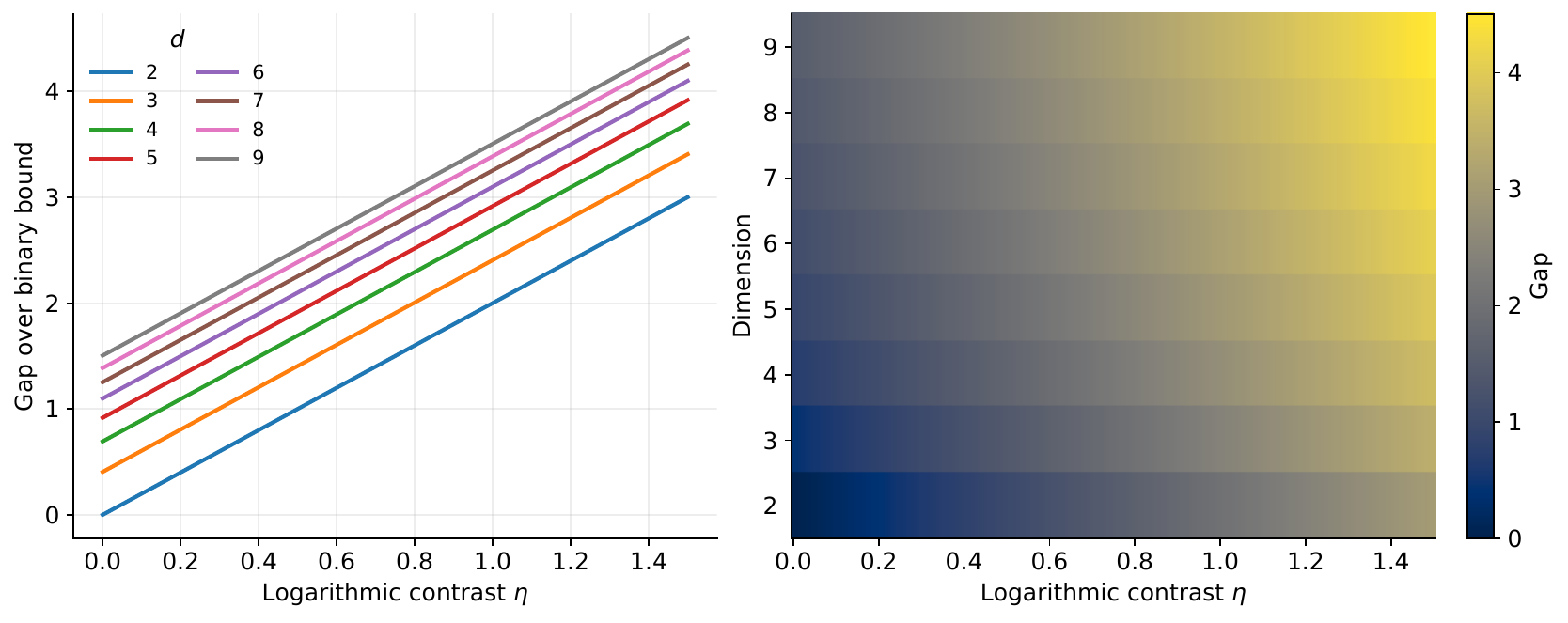}
\caption{Gap between the bound $\ln d+2\eta$ and $\ln2$. Left: curves for $d=2,\ldots,9$. Right: the same gap as a heat map. It vanishes only at $d=2,\eta=0$.}
\label{fig:vacuous-gap}
\end{figure}

\section{Structural Interpretations, Outlook, and Conclusion}

\subsection{Rapidity and Hyperbolic-Geometric Interpretation}

The relation $\QMC(A)=\tanh(\QLC(A))$ also reveals a structural connection with the hyperbolic parametrisation underlying the notion of rapidity in special relativity. In one spatial dimension, a future-directed unit timelike vector in Minkowski space can be written as $u(\eta)=(\cosh\eta,\sinh\eta)$, where $\eta$ is the rapidity, and $\beta=v/c=\tanh\eta$. For collinear boosts, rapidities are additive: $\eta_{12}=\eta_1+\eta_2$.

A formally analogous bounded-to-unbounded parametrisation occurs for OMC: $\QMC(A)=\tanh(\QLC(A))$, $\QLC(A)=\operatorname{arctanh}(\QMC(A))$. Thus OLC can be viewed as a rapidity-type coordinate for the bounded contrast variable OMC. It is important, however, to distinguish this rapidity analogy from the operator-geometric interpretation established earlier. For positive invertible operators, $\QLC(A)=\delta_T([I],[A])$, so OLC is exactly the projective Thompson distance from the identity ray -- a genuine metric identification on the positive cone. By contrast, the connection with rapidity arises from the scalar hyperbolic parametrisation $x=\tanh t$, $t=\operatorname{arctanh} x$, and from the analogy between the corresponding composition structures.

Consequently, the Thompson geometry and the rapidity geometry are distinct geometric structures. The former is associated with the order and spectral structure of the positive operator cone, whereas the latter is associated with Lorentzian geometry and hyperbolic angles in Minkowski space. The common appearance of $\tanh$ does not identify these geometries. This distinction is particularly relevant for the composition law: rapidity is additive for collinear Lorentz boosts, $\eta_{12}=\eta_1+\eta_2$, whereas OLC satisfies only the subadditivity estimate $\QLC(AB)\le\QLC(A)+\QLC(B)$. Exact rapidity additivity is a consequence of the one-dimensional Lorentz boost group, whereas OLC subadditivity follows from the submultiplicativity of the operator condition number; noncommutativity and possible misalignment of extremal spectral directions prevent a general additive law \cite{rindler2006}.

\subsection{Structural Remarks: Thompson Geometry and Harmonic Analysis}

The geometric identification of Theorem \ref{thm:thompson} (OLC as Thompson distance) and the Fourier-theoretic explanation of Corollary \ref{cor:fourier-nonbij} together situate OLC within two complementary analytical frameworks. The Thompson metric governs the contraction and ordering properties of OLC on the positive cone, while the log-spectral Fourier perspective (Section \ref{sec:log-spectral}) explains why functionals of the spectrum (entropy, purity, divergences) are constrained but not determined by OLC in dimensions $\ge3$. These perspectives may deepen through future work connecting them via harmonic analysis on convex cones of log-spectral measures, but the structural foundations are already visible in the present results.

\subsection{Open Questions}\label{sec:open}
\begin{itemize}
\item For a qubit with smaller eigenvalue $q\sim e^{-2\eta}$, R\'enyi entropy $S_\alpha=(1-\alpha)^{-1}\ln((1-q)^\alpha+q^\alpha)$ has tail $q^\alpha/(1-\alpha)$ for $0<\alpha<1$ and $\alpha q/(\alpha-1)$ for $\alpha>1$. Thus the logarithmic slope is $-2\alpha$ and $-2$, respectively. At $\alpha=1$ the additional logarithmic factor is given by Corollary \ref{cor:entropy-asymp}. The sharp R\'enyi envelopes at fixed $\kappa$ for larger dimensions are a separate extremal problem.
\item The tensor result concerns a fixed common decomposition. It is natural to ask how further restrictions on the factors, such as group symmetries or membership in operator systems, constrain attainable equality pairs.
\item Example \ref{ex:entangled-saturation} gives entangled states with additive OLC. A quantitative problem is to determine the maximal entanglement at fixed local dimensions and prescribed marginal or global condition numbers under equality.
\item A comparison with local-to-global estimates in spin chains would require a model identifying the operators and an applicable spectral inequality. The relationship remains open.
\item Lemma \ref{lem:align} settles equality using approximate extremal vectors. For an operator class with additional spectral assumptions, one may ask for quantitative estimates relating the defect in subadditivity to distances between spectral subspaces or to failure of commutativity.
\end{itemize}

\subsection{Speculative Directions}

\subsubsection{Optical Systems and Channel Structure}
Multiplicative responses in optical systems motivate keeping track of ratios of extreme response levels. The following examples describe possible operator models, together with the conditions needed to apply the results.
\begin{enumerate}
\item In the diagonal RGB model of Remark \ref{rem:rgb}, the logarithmic contrast is a spread of the three channel intensities. Such matrices commute because they use the same channel basis. Testing the product equality criterion for coupled channels would require a positive definite model with off-diagonal entries.
\item Multispectral and hyperspectral systems use more than three channels \cite{goetz1985}. A positive definite model with at least four channels can admit noncommuting saturation by Theorem \ref{thm:threshold}. For any chosen model, the common extremal directions must still be verified.
\item Single-qubit polarization coherency matrices act on dimension two. A two-qubit polarization model acts on dimension four and can realize examples such as those in Section \ref{rem:entanglement-section}. In a pure tensor model with two local factors of dimension two, Proposition \ref{prop:tensor-transfer} still forces commutativity at equality. Tensor factorizability and entanglement are distinct restrictions relevant to such models \cite{raussendorf2003,o2003}.
\end{enumerate}
The model of sequential positive filters in Section \ref{sec:filters} gives an operational interpretation when the positive factors can be realized as the successful branch operators of an optical setup.

\subsubsection{Representation Theory and Phase Structure}
The matrix threshold is determined by how much space remains after the common extremal lines have been fixed. A representation-theoretic variant could study the reducing subspaces allowed when the factors are constrained by a given representation or its commutant.

Every unitary has condition number $1$, hence its eigenphases are undetermined by OLC. The transform $\Phi_A(t)$ above is built from a positive operator $A$ and records the distribution of logarithms of its eigenvalues. Applying it to $|A|$ for an arbitrary invertible operator gives the logarithmic singular-value distribution. Recovering unitary eigenphases requires different spectral information. This distinction is relevant when extending the analysis of positive factors to representations involving arbitrary invertible matrices.

\subsubsection{Computational Analogy}

The identity $\QMC(A)=\tanh(\QLC(A))$ admits a computational reading: with $z=\QLC(A)$, $\QMC(A)=\tanh z$ is formally the output of the hyperbolic tangent applied to an unbounded pre-activation coordinate, in the sense familiar from neural-network activation functions \cite{goodfellow2016}. This should be regarded purely as a naming analogy, not as a proposed architecture or as evidence of any deeper connection.

\subsection{Conclusion}
We have studied the logarithmic parametrisation of the operator Michelson contrast, obtaining the subadditive functional $\QLC(A)=\tfrac12\ln\kappa(A)$, which we call the Operator Logarithmic Contrast. This logarithmic lift is the operator analogue of classical logarithmic contrast. We proved its subadditivity, geometric interpretation as a projective Thompson distance, an accompanying Lipschitz-continuity property that also pins down the correct limiting behaviour at singular operators, and analysed the equality case, showing that $n=4$ is the smallest dimension in which noncommuting equality cases can occur -- and that equality forces commutativity for any total dimension when the operators factor in a common pure tensor decomposition into positive factors of dimension at most $3$. The approximate extremal-vector criterion identifies equality on arbitrary Hilbert spaces. Additional properties include contraction, sum inequalities, trace estimates, and estimates for double-sided products $X^{1/2}YX^{1/2}$. Applied to quantum information, these results yield an exact closed-form bijection between OMC and von Neumann entropy for qubits, purity bounds at fixed condition number for qutrits, and a proof -- established by complete analytic arguments rather than numerical checks -- that no bijection between OMC and entropy exists for $d\ge3$, together with sharp qutrit envelopes. In growing dimension, bounded condition number implies unbounded entropy. The logarithmic parametrisation reveals a rich structure inherent in the operator Michelson contrast: not only as a geometric quantity (Thompson distance, Theorem \ref{thm:thompson}), but also as a spectral bandwidth determined by log-spectral distributions (Proposition \ref{prop:fourier-char} and Corollary \ref{cor:fourier-nonbij}, Section \ref{sec:log-spectral}). These foundational perspectives invite future development of finer Fourier-theoretic properties, connections to commutator bounds and Baker--Campbell--Hausdorff theory, and applications to related operator-theoretic problems where the interplay between condition numbers and spectral structure plays a central role.

\appendix
\section{Reproducing the Figures}\label{app:code}

The accompanying Python script \texttt{anc/generate\_figures.py} generates the fifteen PDF figures from the formulas in the text. Run it with Python 3, NumPy and Matplotlib. The script evaluates the small eigenvalue of a qubit directly from $\eta$ for the entropy tail, avoiding the roundoff that occurs if one first rounds $\tanh\eta$ to $1$. It also prints the values for Table \ref{tab:entropy} and the numerical qutrit QJSD example in Remark \ref{rem:bruni}. The marginal extrema in Figures \ref{fig:qjsd_ridges} and \ref{fig:qjsd-marginals} are explicitly computed over the plotted windows.

The first group of plots treats qubit entropy in its bounded and logarithmic coordinates. The second uses the sharp qutrit entropy and purity bounds. The third evaluates QJSD at fixed angle, fixed radius, and over the two-radius parameter window. The final plot computes the gap of Remark \ref{rem:qjsd-vacuous} directly from $\ln d+2\eta-\ln2$.

\paragraph{Preparation of the revision.}
OpenAI ChatGPT assisted with checking formulas and counterexamples, drafting corrections to the revised exposition, and writing the reproducible figure script.

\end{document}